\documentclass[12pt]{article}
\usepackage{graphicx}
\usepackage{amsmath}
\usepackage{amsfonts}
\usepackage{amssymb}
\usepackage{amsthm}
\usepackage[shortlabels]{enumitem}
\usepackage{multicol}
\usepackage[cp1250]{inputenc}
\usepackage[a4paper, left=2.5cm, right=2.5cm, top=3.5cm, bottom=3.5cm, headsep=1.2cm]{geometry}
\usepackage{caption}
\usepackage{color}
\definecolor{sepia}{cmyk}{0, 0.83, 1, 0.70}
\usepackage[plainpages=false,pdfpagelabels,bookmarksnumbered,%
 colorlinks=true,%
 linkcolor=sepia,%
 citecolor=sepia,%
 unicode]{hyperref}

\newtheorem{theorem}{Theorem}
\newtheorem{corollary}[theorem]{Corollary}
\newtheorem{fact}[theorem]{Fact}
\newtheorem{lemma}[theorem]{Lemma}

\theoremstyle{definition}
\newtheorem{definition}[theorem]{Definition}
\newtheorem{example}[theorem]{Example}
\newtheorem{remark}[theorem]{Remark}

\numberwithin{equation}{section} % for equations
\numberwithin{theorem}{section}  % for theorems

\makeatletter
\renewenvironment{proof}[1][\proofname]
{\par
	\pushQED{$\blacksquare$} % originally "\qed" qas instead of $\blacksquare$
	\normalfont\topsep6\p@\@plus6\p@\relax
	\trivlist
	\item[\hskip\labelsep\bfseries#1\@addpunct{.}]
	\ignorespaces}
{\popQED \endtrivlist\@endpefalse}
\makeatother

\DeclareMathOperator*{\fix}{Fix}

\DeclareMathOperator*{\spn}{span}

\DeclareMathOperator*{\rank}{rank}

\begin{document}

\title{\vspace{-5em} \textbf{\Large Comparing the Methods of Alternating and Simultaneous Projections for Two Subspaces}  }

% Alternative titles:
% Comparing the Rates of Convergence for the Methods of Alternating and Simultaneous Projections for two Subspaces

\author{Simeon Reich\thanks{Department of Mathematics,
    The Technion -- Israel Institute of Technology, 3200003 Haifa, Israel;
    E-mail: \texttt{sreich@technion.ac.il}. } \and
    Rafa\l\ Zalas\thanks{Department of Mathematics,
    The Technion -- Israel Institute of Technology, 3200003 Haifa, Israel;
    E-mail: \texttt{rafalz@technion.ac.il}. }
}%author

\maketitle

\begin{abstract}
  We study the well-known methods of alternating and simultaneous projections when applied to two nonorthogonal linear subspaces of a real Euclidean space. Assuming that both of the methods have a common starting point chosen from either one of the subspaces, we show that the method of alternating projections converges significantly faster than the method of simultaneous projections. On the other hand, we provide examples of subspaces and starting points, where the method of simultaneous projections outperforms the method of alternating projections.


  %41A25  	Rate of convergence, degree of approximation
  %41A28  	Simultaneous approximation
  %41A44  	Best constants
  %41A65  	Abstract approximation theory (approximation in normed linear
  %         spaces and other abstract spaces)
\end{abstract}

\section{Introduction}
The best approximation problem for two linear subspaces $A$ and $B$ of $\mathbb R^d$ is to find a point in the intersection $A \cap B$ which is the closest to any given starting point $x \in \mathbb R^d$. Even though the setting of two linear subspaces is very limited, its theoretical study still attracts a lot of attention in both finite and infinite dimensional Hilbert spaces; see, for example, \cite{ArtachoCampoy2019, BauschkeCruzNghiaPhanWang2014, BauschkeCruzNghiaPhanWang2016, BauschkeDeutschHundalPark2003, BehlingCruzSantos2018, FaltGiselsson2017, KayalarWeinert1988, ReichZalas2017} and \cite{Deutsch2001, Galantai2004}. In the context of Banach spaces, see also the recent papers \cite{BargetzKlemencReichSkorokhod2020, BargetzLuggin2023}. For a more detailed overview on the best approximation problem for two or more closed and convex sets and the related convex feasibility problem, we refer the reader to \cite{BauschkeBorwein1996, BauschkeCombettes2017, Cegielski2012, Popa2012}  and \cite{CensorCegielski2015}.

We focus on the Method of Alternating Projections (MAP):
\begin{equation}\label{int:ak}
  a_0 \in \mathbb R^d, \qquad
  a_k:= (P_AP_B)^k (a_0), \qquad
  k=1,2,\ldots,
\end{equation}
and on the Method of Simultaneous Projections (MSP):
\begin{equation}\label{int:xk}
  x_0 \in \mathbb R^d, \qquad
  x_k := \left(\frac{P_A+P_B}{2}\right)^k(x_0), \qquad k=1,2,\ldots,
\end{equation}
where by $P_A$ and $P_B$ we denote the orthogonal projection onto $A$ and $B$, respectively. We recall that these methods converge to $P_{A\cap B}(a_0)$ \cite{Halperin1962, KopeckaReich2004, Neumann1949} and to $P_{A \cap B}(x_0)$ \cite{Lapidus1981, Reich1983}, respectively.  In this paper we try to determine for which starting points $a_0 = x_0$ the MAP is faster than the MSP and when the opposite happens. For this reason we analytically compare the error terms $\|a_k - P_{A \cap B}(a_0)\|$ and $\|x_k - P_{A \cap B}(x_0)\|$, as well as their rates of $Q$-linear convergence;  see Definition \ref{def:Qlinear}.

In order to avoid the trivial case, where the MAP converges in one iteration for all starting points, we assume  that the subspaces are ``nonorthogonal'',  that is, the \emph{Friedrichs angle} $\theta_F = \theta_F(A,B) \in [0, \frac \pi 2)$; see Definition \ref{def:Friedrichs}.  The nonorthogonality condition has a few equivalent forms that will be interchangeably used in the sequel; see Remark \ref{rem:cosANDlambda} and Remark \ref{rem:assumption}.  Since we only consider the Euclidean space $\mathbb R^d$, we can also deduce that $\theta_F > 0$, so that $\cos(\theta_F) \in (0,1)$; see  Remark \ref{rem:parallel}.  We emphasize that for infinite dimensional Hilbert spaces, it might happen that $\cos(\theta_F) = 1$; see, for example, \cite[Example on p. 464]{PustylnikReichZaslavski2010}. We refer the interested reader to \cite{BadeaGrivauxMuller2011, BauschkeBorweinLewis1997, BauschkeDeutschHundal2009, BorodinKopecka2020, DeutschHundal2015, ReichZalas2021, ReichZalas2023}, which discuss such a scenario in more detail.

It is known that the Friedrichs angle determines the (optimal) rates of $R$-linear convergence for the MAP and the MSP, which are given by $\cos^2(\theta_F)$ and $\frac 1 2 + \frac 1 2\cos(\theta_F)$, respectively;  see Definition \ref{def:Rlinear}.  In fact, for the MAP \cite{KayalarWeinert1988}, we have
\begin{equation}\label{int:KW88:eq}
  \|(P_AP_B)^k-P_{A\cap B}\| = \cos^{2k-1}(\theta_F),
\end{equation}
while for the MSP \cite{ReichZalas2017}, we have
\begin{equation} \label{int:RZ17:eq} \textstyle
  \left\|\left(\frac{P_A+P_B} 2\right)^k-P_{A\cap B}\right\| = \left(\frac 1 2 + \frac 1 2\cos(\theta_F)\right)^{k},
\end{equation}
where $k = 1, 2, \ldots$. Note, however, that the above-mentioned rates do not determine the behavior of the particular trajectories $\{a_k\}_{k=0}^\infty$ and $\{x_k\}_{k=0}^\infty$, but rather inform us of the worst-case scenario. In fact, as it will be demonstrated in this paper, when $\theta_F < \frac \pi 3$, we can always find a starting point $a_0 = x_0$ for which the MSP outperforms the MAP in spite of the inequality $\cos^2(\theta_F) < \frac 1 2 + \frac 1 2\cos(\theta_F)$. We return to this below.

This suggests that the information on the Friedrichs angle is not sufficient and that one has to consider other means of describing the relative positions between the subspaces. For this reason, we recall the concept of the \emph{principal angles} $\theta_n \in [0,\frac \pi 2]$ which are recursively defined for $n = 1, \ldots,p$, where $p := \min\{\dim A, \dim B\}$; see, for example, \cite{BjorckGolub1973, Galantai2004} and Definition \ref{def:principal}. The principal angles are increasing and the Friedrichs angle $\theta_F$ happens to be the first nonzero principal angle $\theta_f$, where $f := \dim(A \cap B)+1$. As it turns out \cite{Galantai2008, GalantaiHegedus2006}, for $r := \rank(P_BP_AP_B)$, the scalars $\lambda_n := \cos^2(\theta_n)$, $n = f, \ldots, r$, are the only eigenvalues in $(0,1)$ of the product of projections $P_BP_AP_B$, while the scalars $\mu_n^{+} := \frac 1 2 + \frac 1 2 \cos(\theta_n)$ and $\mu_n^{-} := \frac 1 2 - \frac 1 2 \cos(\theta_n)$, $n = f, \ldots, r$, are the only eigenvalues in $(0,1) \setminus\{\frac 1 2\}$ of the averaged projections $\frac{P_A+P_B}{2}$; see, for example, \cite{BjorstadMandel1991, BottcherSpitkovsky2010, Galantai2008, KnyazevJujunashviliArgentati2010}. We comment on this in Sections \ref{sec:angles} and \ref{sec:eigenvalues}, where we also reestablish these properties.

We note in passing that the Friedrichs angle, as well as the principal angles, play an important role in determining the rate of $R$-linear convergence for other projection methods; see, for example, \cite{ArtachoCampoy2019, BauschkeCruzNghiaPhanWang2014, BauschkeCruzNghiaPhanWang2016, BauschkeDeutschHundalPark2003, BehlingCruzSantos2018, FaltGiselsson2017}. On the other hand, when the projections are nonorthogonal, the above-mentioned angles might not be sufficient \cite{BargetzKlemencReichSkorokhod2020}.

Our interest in the study of the principal angles lies in the fact that the rate of $Q$-linear convergence $\lambda \in (0,1)$ for the sequence $\{\|a_k - P_{A \cap B}(a_0)\|\}_{k=0}^\infty$ happens to be the largest eigenvalue of $P_BP_AP_B$ in $(0,1)$ for which the projection of the point $a_0$ onto the eigenspace $\mathcal N(\lambda - P_BP_AP_B)$ is nonzero. Similarly, the rate of $Q$-linear convergence $\mu \in (0,1)$ for $\{\|x_k - P_{A \cap B}(x_0)\|\}_{k=0}^\infty$ is the eigenvalue of $\frac{P_A+P_B}{2}$ with an analogous property; see Theorem \ref{thm:main}. Even though we precisely identify $\lambda$ and $\mu$ and their dependence on the starting point $a_0 = x_0$, these rates cannot be easily compared in general, and the comparison of the error terms $\|a_k - P_{A \cap B}(a_0)\|$ and $\|x_k - P_{A \cap B}(x_0)\|$ is also not trivial.

\bigskip
The \textbf{main contributions} of our paper are as follows:

\begin{enumerate}[(C1)]
  \item We show that when the starting point $a_0=x_0$ is chosen from either one of the subspaces, then the above-mentioned rates of $Q$-linear convergence satisfy the identity $\mu = \frac 1 2 + \frac 1 2 \sqrt \lambda$. This implies that $\mu > \lambda$. We also demonstrate that in this case the error terms satisfy the inequality $\|x_k - P_{A \cap B}(x_0)\| > \|a_k - P_{A \cap B}(a_0)\|$ for all $k = 1,2,\ldots$. For details, see Theorem \ref{thm:main2}.

  \item In contrast to Theorem \ref{thm:main2}, we show that when the Friedrichs angle $\theta_F < \frac \pi 3$, we can always choose a starting point $a_0 = x_0$ for which the rates satisfy the opposite inequality $\mu < \lambda$. In this case we also have $\|x_k - P_{A \cap B}(x_0)\| < \|a_k - P_{A \cap B}(a_0)\|$ for all $k = 1,2,\ldots$. We comment on this in Remark \ref{rem:cos12}.

  \item Surprisingly, when $\theta_F > \frac \pi 3$, the situation described in (C2) cannot happen and for every starting point $a_0 = x_0 \in \mathbb R^d$ (for which the MAP does not converge in one iteration), we obtain $\lambda < \mu$. We also comment on this in Remark \ref{rem:cos12}.
\end{enumerate}

In order to justify the statements (C1)--(C3), we have conducted a very detailed spectral analysis of the operators $P_BP_AP_B$ and $\frac{P_A+P_B}{2}$. In particular, we provided numerous projection formulas onto their eigenspaces. This allowed us not only to recover the above-mentioned relation between $\lambda_n$, $\mu_n^+$ and $\mu_n^-$, but also to extend this relation to certain subsets of the eigenvalues; see Remark \ref{rem:multiplicity} and Theorem \ref{thm:Lambdax}. We emphasize that Theorem \ref{thm:Lambdax} plays a central role in establishing the identity $\mu = \frac 1 2 + \frac 1 2 \sqrt \lambda$ in (C1).

The outline of the paper is as follows. In Section \ref{sec:spectralTheorem} we introduce the notation and recall the spectral theorem. We also comment on some of the consequences of the spectral theorem that are relevant to our study. In Section \ref{sec:angles} we introduce the Friedrichs and principal angles. The proofs from this section are relegated to \hyperref[sec:AppendixA]{Appendix A}. The above-mentioned spectral analysis is presented Section \ref{sec:spectral}, while the results concerning the rates of convergence appear in Section \ref{sec:rates}.

\section{Preliminaries} \label{sec:preliminaries}
 Let $\langle \cdot, \cdot\rangle $ denote an arbitrary inner product in $\mathbb R^d$ and let $\|\cdot\|$ be the induced norm. We also denote the corresponding operator norm by the same symbol $\|\cdot\|$.

We recall the well-known concepts of $Q$-linear and $R$-linear convergence.

\begin{definition}\label{def:Qlinear}
  We say that a sequence of scalars $\{\alpha_k\}_{k=0}^\infty \subset (0,\infty)$ converges \emph{$Q$-linearly} to zero at the rate $r$, if $r = \lim_{k \to \infty} \alpha_{k+1}/\alpha_k \in (0,1)$.
\end{definition}

 We note that some authors define $Q$-linear convergence with at least the rate $r \in (0, 1)$ if the inequality $\alpha_{k+1} \leq r \alpha_k$ holds for large enough $k$. Both of these notions will be relevant in this manuscript.

\begin{definition}\label{def:Rlinear}
  We say that a sequence of scalars $\{\alpha_k\}_{k=0}^\infty \subset [0,\infty)$  converges  \emph{$R$-linearly} to zero at the rate $r \in (0,1)$, if it can be majorized by another sequence that converges $Q$-linearly to zero at the rate $r$.
\end{definition}

\subsection{Spectral Theorem} \label{sec:spectralTheorem}

Let $T$ be a nonzero and linear operator on $\mathbb R^d$. The null space of $T$ is given by $\mathcal N(T) := \{x \in \mathbb R^d \colon T(x) = 0\}$, where by $0$ we denote the null element of $\mathbb R^d$. Let $I$ denote the identity operator on $\mathbb R^d$. The set of eigenvalues of $T$ is defined by
\begin{equation}\label{def:Lambda}
  \Lambda(T) := \Big\{\lambda \in \mathbb R \colon  \mathcal N (\lambda I - T) \neq \{0\}\Big\}.
\end{equation}
In the sequel, instead of writing $\lambda I$, we simply write $\lambda$ for short. Note that by definition, the set $\Lambda(T)$ consists of distinct elements. However, in some situations it is convenient to repeat the eigenvalues according to their geometric multiplicity. In such a case we write this explicitly using the symbol $\Lambda_{\#}(T)$; see Remarks \ref{rem:SDT} and \ref{rem:multiplicity}, and Lemma \ref{lem:principal}. In the sequel, we will also consider the subset of $\Lambda(T)$ given by
\begin{equation}\label{def:sigmax}
  \Lambda (x, T) := \Big\{\lambda \in \Lambda(T) \colon P_{\mathcal N(\lambda - T)}(x) \neq 0 \Big\}, \quad x \in \mathbb R^d.
\end{equation}
This subset of eigenvalues will play a crucial role in determining the rate of convergence; see Theorem \ref{thm:main}.

Following \cite[Theorem 1, Section 79, p 156]{Halmos1974} and \cite[Theorem 5.12]{Kress2014}, we formulate the spectral theorem in terms of orthogonal projections.

\begin{theorem}[Spectral Theorem]\label{thm:SDT}
  Let $T$ be a nonzero, linear and self-adjoint operator on $\mathbb R^d$. Then, the set $\Lambda(T)$ is nonempty and consists of at least one nonzero element. Moreover, for different eigenvalues $\lambda, \lambda' \in \Lambda(T)\setminus\{0\}$, the corresponding eigenspaces $\mathcal N (\lambda - T)$ and $\mathcal N (\lambda' - T)$, and the null space $\mathcal N(T)$, are pairwise orthogonal. Furthermore, for each $x \in \mathbb R^d$, we have $x = x_0 + \sum_{\lambda \in \Lambda(T)\setminus \{0\}} x_\lambda,$ where $x_0 := P_{\mathcal N(T)}(x)$ and $x_\lambda := P_{\mathcal N (\lambda - T)}(x)$, $\lambda \in \Lambda(T)\setminus \{0\}$. Consequently,   $T^k(x) = \sum_{\lambda \in \Lambda(T)\setminus \{0\}} \lambda^k x_\lambda$, $k = 1,2,\ldots$.
\end{theorem}

\begin{remark} \label{rem:SDT}
  The spectral theorem is oftentimes formulated in an equivalent way. Indeed for $\Lambda_{\#}(T) = \{\lambda_1,\ldots,  \lambda_d  \}$ and the corresponding orthogonal eigenvectors $\{e_1, \ldots,  e_d  \}$, where $T(e_i) = \lambda_i e_i$, we have $x =  \sum_{i=1}^{d}  \langle x, e_i \rangle e_i$. In particular, for each $\lambda \in \Lambda(T)$, the projection $x_\lambda$ of Theorem \ref{thm:SDT} can be written as $x_\lambda = \sum_{\lambda_i = \lambda} \langle x, e_i \rangle e_i$.
\end{remark}

 The representation of $x$ given in Theorem \ref{thm:SDT} is referred to as the \emph{spectral representation of $x$ for the operator $T$}. We note that this representation is unique.
Thus, one can use Theorem \ref{thm:SDT} in order to determine the projections onto $\mathcal N(T)$ and $\mathcal N(\lambda - T)$. Since we refer to this property quite often, we formulate it in the following lemma.

\begin{lemma}[Uniqueness] \label{lem:uniqueness}
  Let $T $ be as in Theorem \ref{thm:SDT}. Moreover, let $x \in \mathbb R^d$ and assume that $x$ can be written as $ x = y_0 + \sum_{\lambda \in \Lambda(T)\setminus \{0\}} y_\lambda$ for some $y_0 \in \mathcal N(T)$ and $y_\lambda \in \mathcal N(\lambda - T)$, $\lambda \in \Lambda(T)\setminus \{0\}$. Then,
  $P_{\mathcal N(T)}(x) = y_0$ and $P_{\mathcal N(\lambda - T)}(x) = y_\lambda$, $\lambda \in \Lambda(T)\setminus \{0\}$. Moreover, $P_{\mathcal N(\lambda - T)}(x) = 0$ for all $\lambda \in \mathbb R \setminus \Lambda(T)$.
\end{lemma}

The next lemma is particularly useful when $\lambda = 1 $, in which case $T_1 = T - P_{\fix T}$.
\begin{lemma}\label{lem:spectralShift}
  Let $T $ be as in Theorem \ref{thm:SDT} and let $T_\lambda := T - \lambda P_{\mathcal N(\lambda - T)}$ for some $\lambda \neq 0$. Then,
  \begin{equation}\label{lem:spectralShift:NandE1}
    \mathcal N(T_\lambda) = \mathcal N(T) \oplus \mathcal N(\lambda - T), \ \
    \mathcal N(\lambda - T_\lambda) = \{0\}
  \end{equation}
  and for each nonzero $\sigma \neq \lambda$, we have
  \begin{equation}\label{lem:spectralShift:NandE2}
    \mathcal N(\sigma - T_\lambda) = \mathcal N(\sigma - T).
  \end{equation}
  In particular,
  \begin{equation}\label{lem:spectralShift:Lambda}
    \Lambda(T_\lambda) \setminus \{0\} = \Lambda(T)\setminus\{0, \lambda\}.
  \end{equation}
  Moreover, for each $x \in \mathbb R^d$ and $x_\lambda := P_{\mathcal N(\lambda - T)}(x)$, we have
  \begin{equation}\label{lem:spectralShift:Lambdax}
    \Lambda (x, T_\lambda) \setminus\{0\} = \Lambda(x,T)\setminus\{0,\lambda\}
    = \Lambda (x - x_\lambda, T) \setminus \{0\}.
  \end{equation}
\end{lemma}

\begin{proof}
  The statement of the lemma is trivial when $\lambda \notin \Lambda (T)$ as $T_\lambda = T$. We may thus assume that $\lambda \in \Lambda(T)\setminus \{0\}$. We first show \eqref{lem:spectralShift:NandE1}--\eqref{lem:spectralShift:NandE2}. Indeed, by Theorem \ref{thm:SDT}, for each $x \in \mathbb R^d$, we have
  \begin{equation}\label{pr:spectralShift:x}
    x = (x_0 + x_\lambda) + \sum_{\varphi \in \Lambda(T)\setminus\{0, \lambda\}} x_\varphi,
  \end{equation}
  where the vectors $x_0 := P_{\mathcal N(T)}(x)$ and $x_\varphi := P_{\mathcal N (\varphi - T)}(x)$, $\varphi \in \Lambda(T)\setminus \{0\}$, are orthogonal. Moreover, for each $\sigma \in \mathbb R$, we get
  \begin{equation}\label{pr:spectralShift:T}
    \sigma x - T_\lambda(x) = \sigma (x_0 + x_\lambda) + \sum_{\varphi \in \Lambda(T)\setminus\{0, \lambda\}} (\sigma - \varphi) x_\varphi.
  \end{equation}
  Assume that $\sigma x - T_\xi(x) = 0$. In each one of the cases, $\sigma = 0$, $\sigma = \lambda$ and $\sigma \in \Lambda(T) \setminus \{0, \lambda\}$, the vectors corresponding to the nonzero scalars on the right-hand side of \eqref{pr:spectralShift:T} must be zero. Thus, we arrive at $x = x_0 + x_\lambda \in \mathcal N(T) \oplus\mathcal N(\lambda - T)$, $x = 0$ and $x = x_\sigma \in \mathcal N(\sigma - T)$, respectively. This shows the inclusions ``$\subset$''. The other direction ``$\supset$'' is obvious. By using the same argument, we see that $\mathcal N(\sigma - T_\lambda) = \{0\}$ for all $\sigma \notin \Lambda(T)$, which completes the proof of \eqref{lem:spectralShift:NandE1}--\eqref{lem:spectralShift:NandE2}.

  It is clear that \eqref{lem:spectralShift:NandE1}--\eqref{lem:spectralShift:NandE2} imply \eqref{lem:spectralShift:Lambda}. Moreover, applying Lemma \ref{lem:uniqueness} to $T_\lambda$ and the representation of $x$ given in \eqref{pr:spectralShift:x}, we get $P_{\mathcal N(T_\lambda)}(x) = x_0+x_\lambda$ and $P_{\mathcal N(\varphi - T_\lambda)}(x) = x_\varphi$ for $\varphi \in \Lambda(T_\lambda) \setminus\{0\}$. Similarly, we have $P_{\mathcal N(T)}(x-x_\lambda) = x_0$, $P_{\mathcal N(\lambda - T)}(x-x_\lambda) = 0$ and $P_{\mathcal N(\varphi - T)}(x-x_\lambda) = x_\varphi$. Note, however, that for $\varphi \in \Lambda(T_\lambda) \setminus\{0\} = \Lambda(T) \setminus\{0, \lambda\}$, the projection $x_\varphi \neq 0$ only when $\varphi \in \Lambda(x,T)\setminus\{0,\lambda\}$. This proves \eqref{lem:spectralShift:Lambdax}.
\end{proof}

\subsection{Dixmier, Friedrichs and Principal Angles} \label{sec:angles}
 Let $A$ and $B$ be two nontrivial ($\neq \{0\}$, $\neq \mathbb R^d$) subspaces of $\mathbb R^d$.  We recall the concept of the Dixmier angle, which is sometimes called the minimal angle.

\begin{definition}\label{def:Dixmier}
  The \emph{Dixmier angle} $\theta_D = \theta_D (A,B) \in [0,\frac \pi 2]$ is implicitly defined through
  \begin{equation} \label{def:Dixmier:eq}
    \cos(\theta_D) = \max\{  \langle a, b \rangle  \colon a\in A,\ b \in B,\  \|a\|, \|b\| \leq 1  \}.
  \end{equation}
\end{definition}

A very similar concept to the Dixmier angle is the Friedrichs angle.
\begin{definition} \label{def:Friedrichs}
  The \emph{Friedrichs angle} $\theta_F = \theta_F (A,B) \in [0,\frac \pi 2]$ is implicitly defined through
  \begin{equation} \label{def:Friedrichs:eq}
    \cos(\theta_F) =
    \max\left\{  \langle a, b \rangle  \colon
      \begin{array}{l}
        a \in A \cap (A \cap B)^\perp, \\
        b \in B \cap (A \cap B)^\perp,
      \end{array}
       \  \|a\|, \|b\| \leq 1
    \right\}.
  \end{equation}
\end{definition}

In general, the two angles can be different. However, when $A \cap B = \{0\}$, then $\theta_F = \theta_D$.  Note that it is enough to consider the maximum in \eqref{def:Dixmier:eq} only on the unit sphere $\|a\| = \|b\| = 1$ and the value will be the same. A similar observation can be made for the maximum in \eqref{def:Friedrichs:eq} when $A \nsubseteq B$ and $B \nsubseteq A$.

\smallskip
We now comment on some less intuitive properties of the Friedrichs angle.

\begin{remark}[Orthogonal and Nonorthogonal Subspaces]
    The Friedrichs angle allows us to define \emph{orthogonal} subspaces through the condition $\cos(\theta_F) = 0$, even when $A \cap B \neq \{0\}$. Moreover, it allows us to define  \emph{nonorthogonal} subspaces via the inequality $\cos(\theta_F) > 0$. Note however, that nonortoghonal subspaces must satisfy $A \nsubseteq B$ and $B \nsubseteq A$. The nonorthogonality condition and its equivalent forms will be oftentimes assumed in this manuscript; see, for example, Remark \ref{rem:cosANDlambda} and Remark \ref{rem:assumption}.
\end{remark}

\begin{remark}[Parallel Subspaces] \label{rem:parallel}
    Surprisingly, the Friedrichs angle is not a suitable tool for defining parallel subspaces in $\mathbb R^d$, as in $\mathbb R^d$ the maximum on the right-hand side in \eqref{def:Friedrichs:eq} is always strictly less than one. In particular, in $\mathbb R^d$, we always have
    \begin{equation}\label{}
      \cos(\theta_F) < 1,
    \end{equation}
    equivalently $\theta_F > 0$. Since this condition is less intuitive, we elaborate on its proof below. We emphasize that for infinite dimensional Hilbert spaces it may happen that $\cos(\theta_F) = 1$ which leads to many interesting behaviours of both the MAP and the MSP. We refer the interested reader to \cite{BadeaGrivauxMuller2011, BauschkeBorweinLewis1997, BauschkeDeutschHundal2009, BorodinKopecka2020, DeutschHundal2015, ReichZalas2021, ReichZalas2023}.
\end{remark}
\begin{proof}
  Suppose that $\cos(\theta) = 1$. Then, using a compactness argument, there are unit vectors $a \in A \cap (A \cap B)^\perp$ and $b \in B \cap (A \cap B)^\perp$ such that
    \begin{equation}\label{}
      1 = \langle a, b\rangle = \langle a, P_A(b)\rangle\leq \|P_A(b)\| \leq \|b\| = 1.
    \end{equation}
    Consequently, $\|P_{A^\perp}(b)\| = 0$ and $b = P_A(b) + P_{A^\perp}(b) = P_A(b) \in A$. This shows that there is a nonzero $b \in A \cap B$ -- in contradiction with $b \in (A \cap B)^\perp$.
\end{proof}

The following lemma connects the Dixmier angle with the largest eigenvalue of the product $P_BP_AP_B$. For the proof, see \hyperref[sec:AppendixA]{Appendix A}.

\begin{lemma}\label{lem:principal0}
  Let $\lambda^* := \max \Lambda(P_BP_AP_B)$. Then, $\cos(\theta_D) = \sqrt{\lambda^*} = \|P_AP_B\|$. Moreover, if the unit vectors $u^* \in A$ and $v^* \in B$ satisfy $\langle u^*, v^*\rangle = \sqrt{\lambda^*}$, then
  \begin{equation}\label{lem:principal0:reciprocal}
    P_A(v^*) = \sqrt{\lambda^*} u^* \quad \text{and} \quad
    P_B(u^*) = \sqrt{\lambda^*} v^*.
  \end{equation}
  In particular, for $\lambda^* >0$, we have $u^* \in \mathcal N(\lambda^* - P_AP_BP_A)$ and $v^* \in \mathcal N(\lambda^* - P_BP_AP_B)$, while for $\lambda^* = 0$, we have $u^* \in A \cap B^\perp$ and $v^* \in B \cap A^\perp$.
\end{lemma}

The following recursive definition is quite common in the literature; see, for example, \cite[Section 1]{BjorckGolub1973}, \cite[p. 25]{Galantai2004}, \cite[Remark 2.6]{BottcherSpitkovsky2010} or \cite[Definition 3.1]{BauschkeCruzNghiaPhanWang2016}. Let $p : = \min\{\dim A, \dim B\}$ and put $u_0 : = v_0 := 0$ in $\mathbb R^d$.

\begin{definition}\label{def:principal}
  The \emph{principal angles} $\theta_n \in [0,\frac \pi 2]$, $n = 1, \ldots, p$, are implicitly defined through $\cos(\theta_n) = \langle u_n, v_n\rangle$, where for each $n = 1, \ldots, p$, the unit vectors $ u_n \in A$ and $v_n \in B$ satisfy
  \begin{equation}\label{def:principal:cos}
    \langle u_n, v_n \rangle = \max \left\{
      \langle u, v \rangle \colon
      \begin{array}{c}
        u \in A,\ v \in B,\ \|u\| = \|v\| = 1 \\
        \langle u, u_i \rangle = \langle v, v_i \rangle = 0,\ i = 0, \ldots, n-1
      \end{array}
    \right\}.
  \end{equation}
\end{definition}

Lemma \ref{lem:principal} ensures that the principal angles do not depend on the choice of the \emph{principal vectors} $u_n$ and $v_n$. Since this result is known, we decided to relegate its proof to \hyperref[sec:AppendixA]{Appendix A}. We mention here only that the proof relies on multiple applications of Lemma \ref{lem:principal0} to properly chosen subspaces $A_n$ and $B_n$.

\begin{lemma}\label{lem:principal}
  Let each one of the tuples $\{u_n\}_{n=1}^p \subset A$ and $\{v_n\}_{n=1}^p \subset B$ consist of orthonormal vectors. Assume that $r := \rank (P_BP_AP_B) \geq 1$. The following conditions are equivalent:
  \begin{enumerate}[(i)]

  \item For each $n = 1,\ldots,p$, the vectors $u_n$ and $v_n$ satisfy \eqref{def:principal:cos}.

  \item There are $\lambda_1 \geq \ldots \geq \lambda_r > 0 = \lambda_{r+1} = \ldots = \lambda_p$ such that
  \begin{equation}\label{lem:principal:reciprocal}
    P_A(v_n) = \sqrt{\lambda_n} u_n \quad \text{and} \quad
    P_B(u_n) = \sqrt{\lambda_n} v_n, \quad n = 1,\ldots,p.
  \end{equation}

  \end{enumerate}
  In any case, (i) or (ii), we get $\langle u_n, v_n\rangle = \sqrt{\lambda_n}$. Moreover, for $n = 1 ,\ldots,r$, we have $u_n \in \mathcal N(\lambda_n - P_AP_BP_A)$ and $v_n \in \mathcal N(\lambda_n - P_BP_AP_B)$, while for $n = r+1, \ldots, p$, we have $u_n \in A \cap B^\perp$ and $v_n \in B \cap A^\perp$. In particular, $\Lambda_{\#}(P_BP_AP_B) \setminus\{0\} = \{\lambda_1, \ldots, \lambda_r\}$.
\end{lemma}

\begin{remark}
  \begin{enumerate}[(i)]
    \item The Dixmier angle $\theta_D$ is the smallest principal angle and coincides with $\theta_1$. On the other hand, using Lemma \ref{lem:principal}, we see that the Friedrichs angle $\theta_F$ is the smallest nonzero principal angle and coincides with $\theta_f$ for $f := \dim(A \cap B)+1$. The proof of the latter statement can be found, for example, in \cite[Proposition 3.3]{BauschkeCruzNghiaPhanWang2016}.

    \item It follows from Lemma \ref{lem:principal} that $\lambda_n = \cos^2(\theta_n)$ are the only nonzero eigenvalues of $P_BP_AP_B$, $n = 1, \ldots, r$. This property can, in fact, be used in order to define the principal angles; see, for example, \cite[Definition 17]{GalantaiHegedus2006} for $\mathbb C^d$ and its generalization \cite[Section 2.2]{KnyazevJujunashviliArgentati2010} to an infinite dimensional Hilbert space. Equality \eqref{lem:principal:reciprocal} is attributed to Afriat \cite{Afriat1957} who called the vectors $u_n$ and $v_n$ \emph{reciprocals}. It can also be found, for example, in \cite[Eq. (6)]{Galantai2008} and in \cite[Eq. (13)]{GalantaiHegedus2006}.
  \end{enumerate}
\end{remark}

\section{Spectral Analysis}
\label{sec:spectral}
 Let again $A$ and $B$ be two nontrivial ($\neq \{0\}$, $\neq \mathbb R^d$) subspaces of $\mathbb R^d$.
In this section we apply the above-mentioned spectral theorem to the two symmetric products of projections $P_AP_BP_A$ and $P_BP_AP_B$, and to the simultaneous projection $\frac{P_A+P_B}{2}$.

\subsection{Nullspaces}
We begin by characterizing the nullspace of $P_BP_AP_B$. Analogously we can characterize the nullspace of $P_AP_BP_A$.

\begin{lemma} \label{lem:NullSet:S}
  We have
  \begin{equation}\label{lem:NullSet:S*S}
    \mathcal N(P_AP_B) = \mathcal N(P_BP_AP_B) = B^\perp \oplus (A^\perp \cap B).
  \end{equation}
\end{lemma}
\begin{proof}
  The inclusion $\mathcal N(P_BP_AP_B) \subset \mathcal N(P_AP_B)$ follows from the identity $\|P_AP_B(x)\|^2 = \langle P_BP_AP_B(x), x\rangle$. The opposite inclusion is obvious.

  The inclusion $B^\perp \oplus (A^\perp \cap B) \subset \mathcal N(P_AP_B)$ follows from the linearity of $P_AP_B$. On the other hand, for $x \in \mathcal N(P_AP_B)$, we get $P_B(x) = P_B(x) - P_AP_B(x) = P_{A^\perp}P_B(x)$. In particular, $P_B(x) \in A^\perp$. By combining this with the orthogonal decomposition theorem, we see that $x = P_{B^\perp}(x) + P_B(x) \in B^\perp \oplus (A^\perp \cap B)$.
\end{proof}

We now characterize the nullspace of the operator $\frac{P_A+P_B}{2}$.

\begin{lemma}\label{lem:NullSet:T}
  We have
  \begin{equation}\label{lem:NullSet:T:N}\textstyle
    \mathcal N(\frac{P_A+P_B}{2}) = A^\perp \cap B^\perp \subset \mathcal N(P_AP_B).
  \end{equation}
\end{lemma}

\begin{proof}
The inclusion  $A^\perp \cap B^\perp \subset \mathcal N(\frac{P_A+P_B}{2})$  is trivial. On the other hand, for $x \in \mathcal N(\frac{P_A+P_B}{2})$,  we have
\begin{equation} \textstyle
  0 = \left\langle x, \frac{P_A+P_B}{2}(x)\right\rangle
  = \frac 1 2 \left\langle x, P_A(x)\right\rangle
  + \frac 1 2 \left\langle x, P_B(x)\right\rangle
  = \frac 1 2 \|P_A(x)\|^2 + \frac 1 2 \|P_B(x)\|^2.
\end{equation}
 Consequently, $P_A(x) = P_B(x) = 0$, which leads to $x \in A^\perp \cap B^\perp$.
\end{proof}

Next, we devote more attention to the eigenspace $\mathcal N(\frac 1 2 - \frac{P_A+P_B}{2})$.

\begin{lemma}\label{lem:EigenSet12:T}
We have
\begin{equation} \label{lem:EigenSet12:T:Set} \textstyle
  \mathcal N(\frac 1 2 - \frac{P_A+P_B}{2}) = (A^\perp \cap B) \oplus (A \cap B^\perp)
  \subset \mathcal N(P_AP_B).
\end{equation}
\end{lemma}

\begin{proof}
It is not difficult to see that for $x = a + b$, where $a \in A\cap B^\perp$ and $b \in B \cap A^\perp$, we get $\frac{P_A+P_B}{2}(x) = \frac 1 2 x$. Consequently, $(A\cap B^\perp) \oplus (B \cap A^\perp) \subset \mathcal N(\frac 1 2 - \frac{P_A+P_B}{2})$. On the other hand, for $x \in \mathcal N(\frac 1 2 - \frac{P_A+P_B}{2})$, we get $x =  P_A(x) +  P_B(x)$. Consequently, $ P_A(x) = x - P_B(x) \in B^\perp$ and $P_B(x) = x - P_A(x) \in A^\perp$. This shows that $x \in (A \cap B^\perp) \oplus (A^\perp \cap B)$ and proves \eqref{lem:EigenSet12:T:Set}. The inclusions of \eqref{lem:EigenSet12:T:Set} follows from Lemma \ref{lem:NullSet:S}.
\end{proof}

We recall the following fact regarding commuting projections from \cite[Lemma 9.2]{Deutsch2001}.

\begin{fact} \label{fact:commutingProj}
  We have $P_A P_B = P_B P_A$ if and only if $P_AP_B = P_{A \cap B}$. In particular, if $A \subset B$, then $P_A P_B = P_B P_A$ but also $P_{A^\perp} P_B = P_BP_{A^\perp}$ (since $P_{A^\perp} = I - P_A$).  Furthermore,
  \begin{equation}\label{}
    P_A P_{A \cap B} = P_{A \cap B} P_A = P_{A \cap B} \quad \text{and} \quad
    P_B P_{A \cap B} = P_{A \cap B} P_B = P_{A \cap B}.
  \end{equation}
\end{fact}

 By combining Lemmata \ref{lem:NullSet:S}--\ref{lem:EigenSet12:T} with Lemma \ref{lem:spectralShift}, we arrive at:
\begin{corollary}\label{cor:NullSets}
  We have
  \begin{equation}\label{cor:NullSets:S}
    \mathcal N(P_AP_B - P_{A \cap B})
    = \mathcal N(P_BP_AP_B - P_{A \cap B})
    = B^\perp \oplus (A^\perp \cap B) \oplus (A \cap B).
  \end{equation}   \
  Moreover,
  \begin{equation}\label{cor:NullSets:T}\textstyle
    \mathcal N(\frac{P_A+P_B}{2} - P_{A \cap B}) = (A^\perp \cap B^\perp) \oplus (A \cap B)
    \subset \mathcal N(P_AP_B - P_{A \cap B})
  \end{equation}
  and
  \begin{equation}\label{cor:NullSets:T12} \textstyle
    \mathcal N(\frac 1 2 - \frac{P_A+P_B}{2}) \subset \mathcal N(P_AP_B - P_{A \cap B}).
  \end{equation}
\end{corollary}

\begin{proof}
  Using Fact \ref{fact:commutingProj}, we have $P_B(P_AP_B - P_{A \cap B}) = P_BP_AP_B - P_{A \cap B}$. This shows the inclusion $\mathcal N(P_AP_B - P_{A \cap B}) \subset \mathcal N(P_BP_AP_B - P_{A \cap B})$. On the other hand, by again referring to Fact \ref{fact:commutingProj}, we have
  \begin{align}\label{} \nonumber
    \|P_AP_B(x) - P_{A \cap B}(x)\|^2
    & =  \langle (P_AP_B - P_{A \cap B})^*\cdot(P_AP_B - P_{A \cap B})(x), x\rangle  \\
    & = \langle P_BP_AP_B(x) - P_{A \cap B}(x), x\rangle,
  \end{align}
  which shows the opposite inclusion. In view of Lemma \ref{lem:spectralShift} and the identity $\mathcal N(1-P_BP_AP_B) = A \cap B$, we get $\mathcal N(P_BP_AP_B - P_{A \cap B}) = \mathcal N(P_BP_AP_B) \oplus (A \cap B)$. Analogously, the identity $\mathcal N(1-\frac{P_A+P_B}{2}) = A \cap B$ leads to $\mathcal N(\frac{P_A+P_B}{2} - P_{A \cap B}) = \mathcal N(\frac{P_A+P_B}{2}) \oplus (A \cap B)$. It now suffices to use Lemmata \ref{lem:NullSet:S}--\ref{lem:EigenSet12:T}.
\end{proof}

\subsection{Orthogonal Projections}

In Lemmata \ref{lem:proj1}--\ref{lem:proj3} we provide formulas for the orthogonal projections onto the nullspaces and eigenspaces of the operators $P_BP_AP_B$ and $\frac{P_A+P_B}{2}$. We emphasize that all of these formulas involve only the basic projections $P_A$ and $P_B$. By switching the roles between $A$ and $B$, we can easily deduce analogous formulas for projections corresponding to $P_AP_BP_A$.

\begin{lemma}\label{lem:proj1}
  Let $\lambda \in (0,1)$ and assume that $u \in \mathcal N(\lambda - P_AP_BP_A)$ is nonzero. Then,
  \begin{equation}\label{lem:proj1:PN}
    P_{\mathcal N(P_BP_AP_B)}(u) = P_{B^\perp}(u) \neq 0
  \end{equation}
  and, for any $\sigma \in (0,1]$, we have
  \begin{equation}\label{lem:proj1:PE}
    P_{\mathcal N(\sigma - P_BP_AP_B)}(u) =
    \begin{cases}
      P_B(u) \neq 0, & \text{if } \sigma = \lambda\\
      0, & \mbox{otherwise}.
    \end{cases}
  \end{equation}
\end{lemma}

\begin{proof}
  We show that the spectral representation of $u$ for $P_BP_AP_B$ takes the form
  \begin{equation}\label{pr:proj1:SDu}
     u\quad = \underbrace{P_{B^\perp}(u)}_{\substack{\in \mathcal N(P_BP_AP_B) \\ \neq 0}} \quad + \quad \underbrace{P_B(u)}_{\substack{\in \mathcal N(\lambda - P_BP_AP_B) \\ \neq 0}}.
  \end{equation}
  It is not difficult to see that $P_BP_AP_B(P_{B^\perp}(u)) = 0$. Knowing that $u = P_A(u)$, we have
  \begin{equation}\label{}
    P_BP_AP_B(P_B(u)) = P_B(P_AP_BP_A(u)) = \lambda P_B(u).
  \end{equation}
  Moreover,
  \begin{equation}\label{pr:proj1:norm}
    \|P_B(u)\|^2 = \|P_BP_A(u)\|^2
    = \langle P_AP_BP_A(u), u\rangle
    = \lambda \|u\|^2 \neq 0.
  \end{equation}
  Since $\lambda < 1$, we also get
  \begin{equation}\label{}
    \|P_{B^\perp}(u)\|^2 = \|u\|^2 - \|P_B(u)\|^2  =  (1-\lambda) \|u\|^2 \neq 0.
  \end{equation}
  It now suffices to apply Lemma \ref{lem:uniqueness} to $P_BP_AP_B$ and \eqref{pr:proj1:SDu}.
\end{proof}

Surprisingly, a similar result holds for the eigenvectors of $\frac{P_A+P_B}{2}$.

\begin{lemma}\label{lem:proj2}
  Let $\mu \in (0,1)$ and assume that $w \in \mathcal N(\mu - \frac{P_A+P_B}{2})$ is nonzero. Then,
  \begin{equation}\label{lem:proj2:PN}
    P_{\mathcal N(P_BP_AP_B)}(w) =
    \begin{cases}
      P_{B^\perp}(w) \neq 0, & \text{if } \mu \neq \frac 1 2\\
      w \neq 0, & \text{if } \mu = \frac 1 2
    \end{cases}
  \end{equation}
  and, for any $\lambda \in (0,1]$, we have
  \begin{equation}\label{lem:proj2:PE}
    P_{\mathcal N(\lambda - P_BP_AP_B)}(w) =
    \begin{cases}
      P_B(w) \neq 0, &  \text{if } \lambda = (2\mu - 1)^2\\
      0,  &  \text{otherwise.}
    \end{cases}
  \end{equation}
\end{lemma}

\begin{proof}
  Assume first that $\mu = \frac 1 2$. Then, Lemma \ref{lem:NullSet:T} leads to $w \in \mathcal N(P_BP_AP_B)$. Consequently, Lemma \ref{lem:uniqueness} yields $P_{\mathcal N(P_BP_AP_B)}(w) = w$ and $P_{\mathcal N(\lambda - P_BP_AP_B)}(w) = 0$.
  Assume now that $\mu \neq \frac 1 2$. We claim that  the spectral representation of $w$ for $P_BP_AP_B$ is given by
  \begin{equation}\label{pr:proj2:SDw}
     w\quad = \underbrace{P_{B^\perp}(w)}_{\substack{\in \mathcal N(P_BP_AP_B) \\ \neq 0}} \quad+ \underbrace{P_B(w)}_{\substack{\in \mathcal N\big((2\mu - 1)^2 - P_BP_AP_B\big) \\ \neq 0}}.
  \end{equation}
  It is not difficult to see that $P_BP_AP_B(P_{B^\perp}(w)) = 0$. On the other hand, by applying $P_A$ and $P_B$ to both sides of the equality
  \begin{equation}\label{pr:proj2:w}
    2 \mu w = P_A(w) + P_B(w),
  \end{equation}
  we obtain
  \begin{equation}\label{pr:proj2:PAPBw}
    P_AP_B(w) =  (2\mu - 1) P_A(w) \quad \text{and} \quad P_BP_A(w) = (2\mu - 1) P_B(w).
  \end{equation}
  Consequently, we get
  \begin{equation}\label{pr:proj2:PBPAPBw}
    P_BP_AP_B(P_B(w)) = (2\mu - 1) P_B P_A(w) = (2\mu - 1)^2 P_B(w).
  \end{equation}
  We show that $P_B(w)$ and $P_{B^\perp}(w)$ are nonzero. Indeed, by \eqref{pr:proj2:PAPBw},
  \begin{equation}\label{pr:proj2:PAwPBw}
    \langle P_A(w), P_B(w)\rangle =
    \begin{cases}
      \langle P_A(w), P_AP_B(w)\rangle = (2\mu-1)\|P_A(w)\|^2 \\
      \langle P_BP_A(w), P_B(w)\rangle = (2\mu-1)\|P_B(w)\|^2.
    \end{cases}
  \end{equation}
  This implies that $\|P_A(w)\| = \|P_B(w)\|$ (since $\mu \neq \frac 1 2$). In fact, by \eqref{pr:proj2:w} and \eqref{pr:proj2:PAwPBw},
  \begin{align}\label{}
    \|2\mu w\|^2 & = \|P_A(w)\|^2 + \|P_B(w)\|^2 + 2 \langle P_A(w), P_B(w)\rangle
    = 4\mu \|P_B(w)\|^2,
  \end{align}
  that is, $\|P_B(w)\| = \sqrt{\mu}\|w\| \neq 0$. Furthermore, since $\mu < 1$, we get
  \begin{equation}\label{}
    \|P_{B^\perp}(w)\|^2 = \|w\|^2 - \|P_B(w)\|^2 =(1 - \mu) \|w\|^2 \neq 0.
  \end{equation}
  It now suffices to apply Lemma \ref{lem:uniqueness} to $P_BP_AP_B$ and \eqref{pr:proj2:SDw}.
\end{proof}

We now describe how to project onto the nullspace and eigenspaces of $\frac{P_A+P_B}2$.

\begin{lemma}\label{lem:proj3}
  Let $\lambda \in (0,1)$ and assume that $u \in \mathcal N(\lambda - P_AP_BP_A)$ is nonzero. Then,
  \begin{equation}\label{lem:proj3:PN}
    P_{\mathcal N\left(\frac{P_A+P_B}2\right)}(u) = 0
  \end{equation}
  and, for any $\mu \in (0,1]$, we have
  \begin{equation}\label{lem:proj3:PE}
    P_{\mathcal N\left(\mu - \frac{P_A+P_B}2\right)}(u) =
    \begin{cases}
      \frac 1 2 u \pm \frac{1}{2\sqrt \lambda} P_B(u) \neq 0, & \text{if } \mu = \frac 1 2 \pm \frac 1 2 \sqrt \lambda \\
      0, & \text{otherwise.}
    \end{cases}
  \end{equation}
\end{lemma}

\begin{proof}
  We show that the spectral  representation of $u$ for  $\frac{P_A+P_B}2$ takes the form
  \begin{equation}\label{pr:proj3:SDv}
     u\quad = \quad \underbrace{\textstyle \frac 1 2 u - \frac {1}{2 \sqrt \lambda} P_B(u)}_{\substack{w_- \in \mathcal N \big(\mu_- - \frac{P_A+P_B}2  \big)\\ \neq 0 }}
    \quad + \quad
    \underbrace{\textstyle \frac 1 2 u + \frac {1}{2 \sqrt \lambda} P_B(u)}_{\substack{w_+ \in \mathcal N \big(\mu_+ - \frac{P_A+P_B}2 \big) \\ \neq 0}},
  \end{equation}
  where $\mu_- := \frac 1 2 - \frac 1 2 \sqrt \lambda$ and $\mu_+ := \frac 1 2 + \frac 1 2 \sqrt \lambda$. Indeed, knowing that $u = P_A(u)$, we obtain $P_AP_B(u) = \lambda u$. A direct calculation shows that
  \begin{align} \nonumber \textstyle
    \frac{P_A+P_B}2(w_-)
    & = \textstyle \frac 1 2 \left( \frac{P_A(u)}{2} - \frac{P_AP_B(u)}{2 \sqrt \lambda}
    + \frac{P_B(u)}{2} - \frac{P_B(u)}{2 \sqrt \lambda} \right) \\ \nonumber
      & =  \textstyle \frac 1 2 \left[\left( \frac 1 2 - \frac{\sqrt\lambda}{2}\right) u
      + \frac{1}{\sqrt{\lambda}}\left(\frac {\sqrt{\lambda}} 2  - \frac 1 2\right)P_B(u) \right]\\
      & =  \textstyle \frac 1 2 \left(  \mu_- \cdot u  - \frac{\mu_-}{\sqrt \lambda} \cdot P_B(u)\right)
      = \mu_- \cdot w_-.
  \end{align}
  Moreover, using $\langle u, P_B(u) \rangle = \|P_B(u)\|^2 = \lambda\|u\|^2$ (see \eqref{pr:proj1:norm}), we have
  \begin{equation}\label{} \textstyle
    \|w_-\|^2 = \frac 1 4 \|u\|^2 + \frac{1}{4\lambda}\|P_B(u)\|^2 - \frac 1 {2 \sqrt{\lambda}} \langle u, P_B(u)\rangle
    = \mu_- \|u\|^2 \neq 0.
  \end{equation}
  Thus we have shown that $w_- \in \mathcal N(\mu_- - \frac{P_A+P_B}2)$ and $w_- \neq 0$, as asserted. By using a similar argument, we can show that $w_+ \in \mathcal N(\mu_+ - \frac{P_A+P_B}2)$ and $w_+ \neq 0$. Thus, by invoking Lemma \ref{lem:uniqueness}, we arrive at \eqref{lem:proj3:PN} and \eqref{lem:proj3:PE}.
\end{proof}

\begin{remark}\label{rem:proj3}
  Alternatively to \eqref{lem:proj3:PE}, we can write
   \begin{equation}\label{rem:proj3:PE}
    P_{\mathcal N\left(\mu - \frac{P_A+P_B}2\right)}(u) =
    \begin{cases}
      \frac 1 2 u + \frac{1}{2(2\mu - 1)} P_B(u) \neq 0, & \text{if } (2\mu-1)^2 = \lambda \\
      0, & \text{otherwise.}
    \end{cases}
  \end{equation}
\end{remark}

In the next lemma we show that in some cases, the nontrivial projections presented in Lemmata \ref{lem:proj1}--\ref{lem:proj3} preserve the orthogonality of eigenvectors.

\begin{lemma} \label{lem:orthogonality}
  Let $\lambda, \mu \in (0,1)$, $\mu \neq \frac 1 2$. Moreover, let $u_1, u_2 \in \mathcal N(\lambda - P_AP_BP_A)$ and let $w_1, w_2 \in \mathcal N(\mu - \frac{P_A+P_B}2)$. Then,
  \begin{equation}\label{lem:orthogonality:1}
    \langle P_B(u_1), P_B(u_2) \rangle = \lambda \langle u_1, u_2\rangle,
  \end{equation}
  \begin{equation}\label{lem:orthogonality:2}
    \langle P_B(w_1), P_B(w_2) \rangle = \mu \langle w_1, w_2\rangle
  \end{equation}
  and
  \begin{equation}\label{lem:orthogonality:3} \textstyle
    \left\langle \frac 1 2 u_1 \pm \frac {1}{2 \sqrt \lambda} P_B(u_1), \frac 1 2 u_2 \pm\frac {1}{2 \sqrt \lambda} P_B(u_2)\right\rangle
    = (\frac 1 2 \pm \frac 1 2 \sqrt \lambda) \langle u_1, u_2\rangle.
  \end{equation}
\end{lemma}

\begin{proof}
  Knowing that $u_1, u_2 \in A$,  we have
  \begin{equation}\label{pr:orthogonality:1}
    \langle P_B(u_1), P_B(u_2)\rangle
    = \langle P_BP_A(u_1), P_BP_A(u_2)\rangle
    = \langle P_AP_BP_A(u_1), u_2\rangle
    = \lambda \langle u_1, u_2\rangle.
  \end{equation}
  By Lemma \ref{lem:proj2}, we see that $P_B(w_2) \in \mathcal N((2\mu-1)^2 - P_BP_AP_B)$. Moreover, thanks to \eqref{pr:proj2:PAPBw}, we have $P_AP_B(w_2) = (2\mu-1)P_A(w_2)$. Thus, by  switching the roles of $A$ and $B$ in  Lemma \ref{lem:proj3} (see also Remark \ref{rem:proj3}), we obtain
  \begin{align}\label{} \nonumber
     \langle P_B(w_1), P_B(w_2) \rangle
    & = \textstyle \Big \langle  w_1, P_{\mathcal N \big(\mu - \frac{P_A+P_B}{2}\big)}(P_B(w_2)) \Big\rangle \\ \nonumber
    & = \textstyle \left\langle w_1, \frac 1 2 P_B(w_2) + \frac{1}{2(2\mu-1)} P_A(P_B(w_2)) \right\rangle\\ \nonumber
    & = \textstyle \left\langle w_1, \frac 1 2 P_B(w_2) + \frac 1 2 P_A(w_2) \right\rangle \\
    & = \mu \langle w_1, w_2 \rangle.
  \end{align}
  Finally, using the identity $\langle u_1,  P_B(u_2)  \rangle = \langle P_B(u_1),  P_B(u_2)  \rangle$ and \eqref{pr:orthogonality:1}, we arrive at
  \begin{align}\label{}\nonumber \textstyle
    \left\langle \frac 1 2 u_1 + \frac {1}{2 \sqrt \lambda} P_B(u_1), \frac 1 2 u_2 +\frac {1}{2 \sqrt \lambda} P_B(u_2)\right\rangle
    & = \textstyle \frac 1 4 \langle u_1, u_2\rangle + \left(\frac{1}{4\lambda} + \frac{1}{2\sqrt \lambda}\right) \langle P_B(u_1), P_B(u_2)\rangle \\ \nonumber
    & = \textstyle \frac 1 4 \langle u_1, u_2\rangle + \left(\frac{1}{4} + \frac{\sqrt \lambda}{2}\right) \langle u_1, u_2\rangle \\
    & = \textstyle \left(\frac{1}{2} + \frac{\sqrt \lambda}{2}\right) \langle u_1,u_1\rangle.
  \end{align}
  A similar argument can be used for showing \eqref{lem:orthogonality:3} with a minus ``$-$'' sign.
\end{proof}

\begin{remark}
  In Lemmata \ref{lem:proj1}--\ref{lem:proj3} we omitted the eigenvectors $u$ and $w$ corresponding to the eigenvalues $\lambda = \mu = 1$. Note, however, that for each $x \in A \cap B$ and for any $\sigma \in [0,1]$, we get
  \begin{equation}\label{}
    P_{\mathcal N(\sigma - P_BP_AP_B)}(x)
    = P_{\mathcal N\big(\sigma - \frac{P_A+P_B}{2}\big)}(x)
    = \begin{cases}
        x, & \text{if } \sigma = 1\\
        0, & \text{otherwise.}
      \end{cases}
  \end{equation}
\end{remark}

\begin{remark}\label{rem:cosANDlambda}
Observe that any of the conditions:
\begin{enumerate}[(i)]
  \item There is an eigenvalue $\lambda \in (0,1)$ of $P_AP_BP_A$ (see Lemmata \ref{lem:proj1} and \ref{lem:proj3});
  \item There is an eigenvalue $\mu \in (0,1)\setminus \{\frac 1 2\}$ of $\frac{P_A+P_B}{2}$ (see Lemma \ref{lem:proj2});
\end{enumerate}
implies the inequality $\cos(\theta_F) > 0$. In fact, as will be demonstrated in Theorem \ref{thm:Lambda} below, the inequality $\cos(\theta_F) > 0$ also implies (i) and (ii).
\end{remark}

\begin{proof}
Let $u \in \mathcal N(\lambda - P_AP_BP_A)$ be nonzero where $\lambda \in (0,1)$. Using the identities $P_AP_BP_A(u) = P_AP_B(u)$ and $P_{A \cap B}(u) = 0$, and by \eqref{int:KW88:eq}, we get
\begin{equation}\label{pr:cosANDlambda}
  0 < \lambda \|u\| = \|P_AP_BP_A(u)\| = \|(P_AP_B - P_{A \cap B})(u)\| \leq \cos(\theta_F) \|u\|.
\end{equation}
This shows that $\cos(\theta_F) > 0$ when (i) holds. Similarly, let $w \in \mathcal N(\mu - \frac{P_A+P_B}{2})$ be nonzero for some $\mu \in (0,1)\setminus \{\frac 1 2\}$. Then, by switching the roles of $A$ and $B$ in \eqref{lem:proj2:PE}, we conclude that $u := P_A(w) \in \mathcal N(\lambda - P_AP_BP_A)$ is nonzero, where $\lambda = (2\mu - 1) \in (0,1)$. From this point we can again use \eqref{pr:cosANDlambda} to see that $\cos(\theta_F) > 0$, when (ii) holds.
\end{proof}

\subsection{Eigenvalues} \label{sec:eigenvalues}
Below we present a connection between $\Lambda(P_BP_AP_B)$ and $\Lambda(\frac{P_A+P_B}2)$.

\begin{theorem} \label{thm:Lambda}
  Assume that $\cos(\theta_F) > 0$. Then,
  \begin{equation}\label{thm:Lambda:TandS} \textstyle
    \Lambda\left(\frac{P_A+P_B}2\right)\setminus \left\{0, \frac 1 2, 1 \right\}
    = \left\{\frac 1 2 \pm \frac 1 2 \sqrt \lambda \colon \lambda \in \Lambda(P_BP_AP_B)\setminus \{0,1\} \right\} \neq \emptyset
  \end{equation}
  with the maximum $\frac 1 2 + \frac 1 2 \cos(\theta_F) < 1$ and with the minimum $\frac 1 2 - \frac 1 2 \cos(\theta_F) > 0$.
\end{theorem}

\begin{proof}
Using \eqref{int:KW88:eq}, Fact \ref{fact:commutingProj}, \eqref{lem:spectralShift:Lambda} and the identity $A \cap B = \mathcal N(1 - P_BP_AP_B)$, we find that
  \begin{align} \label{pr:Lambda:cos}\nonumber
    \cos(\theta_F)^2 &
    = \|(P_AP_B - P_{A\cap B})^*\| \cdot \|P_AP_B - P_{A\cap B}\| \\ \nonumber
    & = \|(P_AP_B - P_{A\cap B})^* \cdot ( P_AP_B  - P_{A\cap B})\| \\ \nonumber
    & = \|P_BP_AP_B - P_{A\cap B}\| \\ \nonumber
    & = \max\Lambda(P_BP_AP_B - P_{A\cap B}) \\
    & = \max\big\{\Lambda(P_BP_AP_B)\setminus\{0,1\}\big\} \in (0,1).
  \end{align}
This shows that $\Lambda(P_BP_AP_B)\setminus \{0,1\} \neq \emptyset$ with the maximum $\cos^2(\theta_F) < 1$;  see also Remark \ref{rem:parallel}.

Denote the set on the right-hand side of \eqref{thm:Lambda:TandS} by $\Omega$ and let $\mu \in \Omega$, say $\mu = \frac 1 2 + \frac 1 2 \sqrt{\lambda}$ for some $\lambda \in \Lambda (P_BP_AP_B) \setminus \{0,1\}$ (a similar argument holds for $\mu = \frac 1 2 - \frac 1 2 \sqrt{\lambda}$). Then, there is a nonzero eigenvector $v \in \mathcal N(\lambda - P_BP_AP_B)$ and, thanks to Lemma \ref{lem:proj3}, we know that $w := \frac 1 2 v + \frac{1}{2\sqrt{\lambda}}P_A(v) \in \mathcal N(\mu - \frac{P_A+P_B} 2)$ is nonzero. In particular, $\mu \in \Lambda(\frac{P_A+P_B}{2})\setminus\{0,\frac 1 2, 1\}$. Thus we have shown that $\Omega \subset \Lambda(\frac{P_A+P_B}{2})\setminus\{0,\frac 1 2, 1\}$.

Assume now that $\mu \in \Lambda(\frac{P_A+P_B}{2})\setminus\{0,\frac 1 2, 1\}$. Then there is a nonzero eigenvector $w \in \mathcal N(\mu - \frac{P_A+P_B}{2})$ and, thanks to Lemma \ref{lem:proj2}, we know that $v := P_B(w) \in \mathcal N(\lambda - P_BP_AP_B)$ is nonzero for $\lambda = (2\mu - 1)^2$. In particular, $\mu = \frac 1 2 \pm \frac 1 2 \sqrt{\lambda} \in \Omega$. Consequently, $\Lambda(\frac{P_A+P_B}{2})\setminus\{0,\frac 1 2, 1\} \subset \Omega$.
\end{proof}

\begin{corollary}\label{cor:Lambda}
  Assume that $\cos(\theta_F) > 0$. Then,
  \begin{equation}\label{cor:Lambda:SandT}
    \Lambda(P_BP_AP_B) \setminus \{0, 1\}
    = \textstyle \{(2\mu - 1)^2 \colon \mu \in \Lambda(\frac{P_A+P_B}{2}) \setminus \{0, \frac 1 2 ,1\}\} \neq \emptyset
  \end{equation}
  with the maximum $\cos^2(\theta_F) \in (0,1)$.
\end{corollary}

It is not difficult to see that we can replace the product $P_BP_AP_B$ by $P_AP_BP_A$ in Theorem \ref{thm:Lambda} and Corollary \ref{cor:Lambda}. In particular,
\begin{equation}\label{rem:multiplicity:SandS'}
  \Lambda(P_BP_AP_B) \setminus \{0, 1\} = \Lambda(P_AP_BP_A) \setminus \{0, 1\} \neq \emptyset.
\end{equation}

\begin{remark}[Including Geometric Multiplicity and Principal Angles] \label{rem:multiplicity}
Thanks to Lemma \ref{lem:orthogonality}, for all $\lambda \in (0,1)$ and $\mu_\pm := \frac 1 2 \pm \frac 1 2 \sqrt{\lambda}$, we have
\begin{equation}\label{}\textstyle
  \dim \mathcal N(\lambda - P_BP_AP_B) = \dim \mathcal N(\lambda - P_AP_BP_A) = \dim \mathcal N(\mu_\pm - \frac{P_A+P_B}{2}).
\end{equation}
This implies that we can rewrite \eqref{thm:Lambda:TandS} and \eqref{rem:multiplicity:SandS'} using ``$\Lambda_{\#}(\cdot)$'' instead of ``$\Lambda(\cdot)$'' so that the eigenvalues are repeated according to their geometric multiplicity. We recall from Section \ref{sec:angles}, that $\Lambda_{\#}(P_BP_AP_B) \setminus \{0,1\} = \{\cos^2(\theta_n) \colon n = f, \ldots,r\}$, where $\theta_n$ are the principal angles, $f = \dim (A \cap B)+1$ and $r = \rank (P_BP_AP_B)$. Thus, $\Lambda_{\#}(\frac{P_A+P_B} 2) \setminus \{0,\frac 1 2, 1\} = \{\frac 1 2 \pm \frac 1 2 \cos(\theta_n) \colon n = f, \ldots, r\}$. More results of this type concerning the spectra of $P_A+P_B$ and $P_BP_AP_B$ can be found, for example, in \cite[Corollary 4.9]{BjorstadMandel1991}, \cite[Theorem 23]{Galantai2008}, \cite[Example 2.1]{BottcherSpitkovsky2010} and in \cite[Theorem 2.18]{KnyazevJujunashviliArgentati2010}.  Although these results were established in different settings, when formulated in our notation, they can be written as follows: $\Lambda_{\#}(P_A+P_B) \setminus \{0, 1, 2\} = \{1 \pm  \cos(\theta_n) \colon n = f, \ldots, r\}.$

\end{remark}

\smallskip
In the next theorem we show that \eqref{thm:Lambda:TandS} also holds when we use ``$\Lambda(x, \cdot )$'' instead of ``$\Lambda(\cdot)$''; see the notation of \eqref{def:sigmax}. We emphasize that such a result is not true for all $x \in \mathbb R^d$, as we show in Example \ref{ex:Lambdax} below. Before proceeding, we note that:

\begin{remark}\label{rem:assumption}
  If $\|P_AP_B(x) - P_{A \cap B}(x)\| > 0$ for some $x \in \mathbb R^d$, then $\cos(\theta_F) > 0$. On the other hand, if $\cos(\theta_F) > 0$, then there is $x \in A \cup B$ such that $\|P_AP_B(x) - P_{A \cap B}(x)\| > 0$. Moreover, if $x \in A \cup B$, then the following conditions are all equivalent:
  \begin{multicols}{2}
  \begin{enumerate}[(i)]
    \item $\|P_AP_B(x) - P_{A \cap B}(x)\| > 0$;
    \item $\|P_BP_A(x) - P_{A \cap B}(x)\| > 0$;
    \item $\|P_BP_AP_B(x) - P_{A \cap B}(x)\| > 0$;
    \item $\|P_AP_BP_A(x) - P_{A \cap B}(x)\| > 0$.
  \end{enumerate}
  \end{multicols}
\end{remark}

\begin{proof}
  The inequality $\|P_AP_B(x) - P_{A \cap B}(x)\| > 0$ implies that the corresponding operator $P_AP_B - P_{A \cap B}$ is nonzero. This, when combined with \eqref{int:KW88:eq}, leads to $\cos(\theta_F) > 0$.

  For the second statement, let $\lambda \in \Lambda(P_AP_BP_A) \setminus \{0,1\}$. The latter set is nonempty thanks to \eqref{rem:multiplicity:SandS'}. Let $u \in \mathcal N(\lambda - P_AP_BP_A) \subset A$ be nonzero and put $v := P_B(u) \in B$. Then $P_{A \cap B}(u) = P_{A \cap B}(v) = 0$, but $P_AP_B(v) = P_AP_B(u) = \lambda u \neq 0$.

  In view of Lemma \ref{lem:NullSet:S}, (i) is equivalent to (iii). Analogously, by switching the roles of $A$ and $B$, we see that (ii) is equivalent to (iv). Moreover, if $x \in A $, then $P_AP_B(x) = P_AP_BP_A(x)$, which shows that (i) is equivalent to (iv). Using an analogous argument, we can show that for $x \in B$, condition (ii) is equivalent to (iii).
\end{proof}

\begin{theorem} \label{thm:Lambdax}
  Let $x \in A \cup B$ be such that $\|P_AP_B(x) - P_{A \cap B}(x)\| > 0$. Then,
  \begin{equation} \label{thm:Lambdax:TandS} \textstyle
    \Lambda\left(x, \frac{P_A+P_B}2 \right)\setminus \left\{0, \frac 1 2, 1 \right\} = \left\{\frac 1 2 \pm \frac 1 2 \sqrt \lambda \colon \lambda \in \Lambda(x, P_BP_AP_B)\setminus \{0,1\} \right\} \neq \emptyset.
  \end{equation}
\end{theorem}

\begin{proof}
  We assume that  $x \in A$.  Using the spectral theorem (Theorem \ref{thm:SDT}), we have
  \begin{equation}\label{pr:Lambdax:x}
    x = u_0 + \sum_{\lambda \in \Lambda(P_AP_BP_A) \setminus \{0,1\}} u_{\lambda} + u_1,
  \end{equation}
  where $u_0 := P_{\mathcal N(P_AP_BP_A)}(x)$ and $u_{\lambda} := P_{\mathcal N(\lambda - P_AP_BP_A)}(x)$, $\lambda \in \Lambda(P_AP_BP_A) \setminus \{0\}$. Note that $u_1 = P_{A\cap B}(x)$. Thanks to Remark \ref{rem:assumption}, we have $\|P_AP_BP_A(x) - P_{A \cap B}(x)\| > 0$. Using the orthogonality of the eigenvectors, we obtain
  \begin{equation}\label{}
    0 < \|P_AP_BP_A(x) - P_{A \cap B}(x)\|^2
    = \sum_{\lambda \in \Lambda(P_AP_BP_A) \setminus \{0,1\}} \lambda^2 \|u_{\lambda}\|^2.
  \end{equation}
  This shows that $\Lambda(x,P_AP_BP_A) \setminus \{0, 1\} \neq \emptyset$, say it consists of $\{\lambda_1 , \ldots , \lambda_m\}$ for some $m \geq 1$.

  Lemma \ref{lem:NullSet:S} applied to $P_AP_BP_A$ implies that $u_0 \in A^\perp \oplus (A \cap B^\perp)$. Moreover, $u_{\lambda_i} \in A$, $i = 1, \ldots, m$.  Thus, the equality $x = P_A(x)$, when combined with \eqref{pr:Lambdax:x}, implies that $u_0 = P_A(u_0)$. In particular, $u_0 \in A \cap B^\perp $ and, thanks to Lemma \ref{lem:EigenSet12:T}, we get $u_0 \in \mathcal N(\frac 1 2 - \frac {P_A+P_B}2)$.
  On the other hand, using Lemma \ref{lem:proj3}, for each $i = 1, \ldots, m$, we obtain
  \begin{equation}\label{pr:Lmabdax:ui}
     u_{\lambda_i}\quad
     = \quad \underbrace{\textstyle \frac 1 2 u_{\lambda_i} - \frac {1}{2 \sqrt {\lambda_i}} P_B(u_{\lambda_i})}_{ \substack{ w_i^-  \in \mathcal N \big(\mu_i^- - \frac{P_A+P_B}2  \big)\\ \neq 0 }}
    \quad + \quad
    \underbrace{\textstyle \frac 1 2 u_{\lambda_i} + \frac {1}{2 \sqrt {\lambda_i}} P_B(u_{\lambda_i})}_{\substack{w_i^+ \in \mathcal N \big(\mu_i^+ - \frac{P_A+P_B}2 \big) \\ \neq 0}},
  \end{equation}
  where $\mu_i^- := \frac 1 2 - \frac 1 2 \sqrt {\lambda_i}$ and $\mu_i^+ := \frac 1 2 + \frac 1 2 \sqrt {\lambda_i}$.

  Put $w_{\frac 1 2 } : = u_0$ and $w_1 := u_1$. We can now write
  \begin{equation} \label{pr:Lambdax:xA}
    x = w_{\frac 1 2} + \sum_{i=1}^m (w_i^- + w_i^+) + w_1,
  \end{equation}
  where the summands $w_{\frac 1 2}$, $w_i^-$, $w_i^+$ and $w_1$ belong to different eigenspaces $\mathcal N(\mu - \frac{P_A+P_B}2)$ with $\mu$ equal to $\frac 1 2$, $\mu_i^-$, $\mu_i^+$ and $1$, respectively, $i = 1, \ldots, m$. Moreover, $w_i^-$ and $w_i^+$ are nonzero, $i = 1, \ldots, m$. Using Lemma \ref{lem:uniqueness}, we obtain
  \begin{equation} \label{pr:Lambdax:mupm}\textstyle
    \Lambda(x, \frac{P_A+P_B}2) \setminus\{0,\frac 1 2, 1\}
    = \{\mu_1^-, \ldots \mu_m^-,\ \mu_1^+, \ldots,\mu_m^+\}.
  \end{equation}

  Clearly, $P_{B^\perp}(x) \in \mathcal N(P_BP_AP_B)$. Moreover, by Lemma \ref{lem:proj1}, $P_B(u_i) \in \mathcal N(\lambda_i - P_BP_AP_B)$ are nonzero, $i = 1,\ldots,m$. Furthermore, knowing that $u_0 \in A \cap B^\perp$, we have $P_B(u_0) = 0$. Thus, we can write
  \begin{equation}
    x = P_{B^\perp}(x) + P_B(x)
    = P_{B^\perp}(x) + \sum_{i=1}^m P_B(u_{\lambda_i}) + u_1,
  \end{equation}
  where again each summand, $P_{B^\perp}(x)$, $P_B(u_{\lambda_i})$ and $u_1$, belong to a different eigenspace $\mathcal N(\lambda - P_BP_AP_B)$ with $\lambda$ equal to $0$, $\lambda_i$ and $1$, respectively, $i = 1,\ldots,m$. Consequently, using Lemma \ref{lem:uniqueness}, we find that
  \begin{equation} \label{pr:Lambdax:lambda}
    \Lambda(x, P_BP_AP_B) \setminus \{0,1\} = \{\lambda_1, \ldots, \lambda_m\}.
  \end{equation}
  By combining \eqref{pr:Lambdax:mupm} with \eqref{pr:Lambdax:lambda}, we arrive at \eqref{thm:Lambdax:TandS}.

  \bigskip
  Assume now that $x \in B$. Thanks to Remark \ref{rem:assumption}, we have $\|P_BP_AP_B(x) - P_{A \cap B}(x)\| > 0$. We can thus repeat the argument from the previous case by simply switching the roles of the subspaces $A$ and $B$. For the convenience of the reader, we highlight the main steps.
  The subset $\Lambda(a_0, P_BP_AP_B) \setminus \{0,1\}$ is nonempty and consists of $\{\lambda_1, \ldots, \lambda_m\}$. Moreover,
  \begin{equation}\label{}
    x = v_0 + \sum_{i=1}^{m} v_{\lambda_i} + v_1,
  \end{equation}
  where $v_0 := P_{\mathcal N(P_BP_AP_B)}(x)$, $v_{\lambda_i} := P_{\mathcal N(\lambda_i - P_BP_AP_B)}(x) \neq 0$, $i = 1,\ldots,m$ and $v_1 \in A \cap B$. The assumption $a_0 \in B$ implies that $v_0 \in A^\perp \cap B$. Using an analogue of Lemma \ref{lem:proj3}, for each $i = 1, \ldots, m$, we get
  \begin{equation}\label{pr:Lmabdax:vi}
     v_{\lambda_i}\quad
     = \quad \underbrace{\textstyle \frac 1 2 v_{\lambda_i} - \frac {1}{2 \sqrt {\lambda_i}} P_A(v_{\lambda_i})}_{\substack{w_i^-  \in \mathcal N \big(\mu_i^- - \frac{P_A+P_B}2  \big)\\ \neq 0 }}
    \quad + \quad
    \underbrace{\textstyle \frac 1 2 v_{\lambda_i} + \frac {1}{2 \sqrt {\lambda_i}} P_A(v_{\lambda_i})}_{\substack{w_i^+ \in \mathcal N \big(\mu_i^+ - \frac{P_A+P_B}2 \big) \\ \neq 0}},
  \end{equation}
  where $\mu_i^- := \frac 1 2 - \frac 1 2 \sqrt {\lambda_i}$ and $\mu_i^+ := \frac 1 2 + \frac 1 2 \sqrt {\lambda_i}$. By putting $ w_{\frac 1 2 } : = v_0$ and $ w_1 := v_1$, we can again write
  \begin{equation} \label{pr:Lambdax:xB}
    x = w_{\frac 1 2} + \sum_{i=1}^m (w_i^- + w_i^+) + w_1,
  \end{equation}
  where the summands $w_{\frac 1 2}$, $w_i^-$, $w_i^+$ and $w_1$ belong to different eigenspaces $\mathcal N(\mu - \frac{P_A+P_B}2)$ with $\mu$ equal to $\frac 1 2$, $\mu_i^-$, $\mu_i^+$ and $1$, respectively, $i = 1, \ldots, m$. In particular, using Lemma \ref{lem:uniqueness}, we obtain an analogue of \eqref{pr:Lambdax:mupm}. We can also show that $\Lambda(x, P_AP_BP_A)\setminus\{0,1\} = \{\lambda_1, \ldots, \lambda_m\}$.
\end{proof}

Below we present an immediate consequence of Theorem \ref{thm:Lambdax}, which is an analogue of Corollary \ref{cor:Lambda}.

\begin{corollary} \label{cor:Lambdax}
  Let $x \in A \cup B$ be such that $\|P_AP_B(x) - P_{A \cap B}(x)\| > 0$. Then,
  \begin{equation}\label{cor:Lambdax:SandS'} \textstyle
    \Lambda(x,P_BP_AP_B) \setminus \{0,1\}
    = \{(2\mu - 1)^2 \colon \mu \in \Lambda(x, \frac{P_A+P_B}{2}) \setminus \{0, \frac 1 2 ,1\}\} \neq \emptyset.
  \end{equation}
\end{corollary}

Clearly, we can again replace the product $P_BP_AP_B$ by $P_AP_BP_A$ in Theorem \ref{thm:Lambdax} and in Corollary \ref{cor:Lambdax}. In particular, analogously to \eqref{rem:multiplicity:SandS'}, for $x \in A \cup B$ with $ \|P_AP_B(x) - P_{A \cap B}(x)\| > 0 $, we get
  \begin{equation}\label{rem:Lambdax:SandS'}
    \Lambda(x,P_BP_AP_B) \setminus \{0, 1\} = \Lambda(x,P_AP_BP_A) \setminus \{0, 1\}.
  \end{equation}
In the example below, we show that equalities \eqref{thm:Lambdax:TandS}, \eqref{cor:Lambdax:SandS'} and \eqref{rem:Lambdax:SandS'} do not hold for all $x \in \mathbb R^d$.

\begin{example} \label{ex:Lambdax}
  Let $\lambda \in (0,1)$ and assume that $u \in \mathcal N(\lambda - P_AP_BP_A)$ is nonzero. Put $x := P_{B^\perp}(u)$. We claim that for $\mu_\pm:= \frac 1 2 \pm \frac 1 2 \sqrt{\lambda}$, we have
  \begin{equation}\label{ex:Lambdax:eq1} \textstyle
    \Lambda(x, P_BP_AP_B) = \{0\},\
    \Lambda(x, P_AP_BP_A) = \{0, \lambda\} \  \text{and} \
    \Lambda(x, \frac{P_A+P_B}2) = \{\mu_+,\mu_-\}.
  \end{equation}
\end{example}

\begin{proof}
  It is clear that $x \in \mathcal N(P_BP_AP_B)$ and $x \neq 0$; see \eqref{lem:proj1:PN}. Consequently, $\Lambda(x, P_BP_AP_B) = \{0\}$. On the other hand, by Lemma \ref{lem:proj1} (and its analogoue, where $A \leftrightarrow B$), the vectors $P_B(u) \in \mathcal N(\lambda - P_BP_AP_B)$ and $P_{A^\perp}P_B(u) \in \mathcal N(P_AP_BP_A)$ are both nonzero. Using the orthogonal decomposition theorem, we arrive at the following representation of $x$  for $P_AP_BP_A$:
  \begin{align}\label{} \nonumber
    x & = u - P_B(u) = u - (P_AP_B(u) + P_{A^\perp}P_B(u)) \\ \nonumber
    & = u - P_AP_BP_A(u)  -  P_{A^\perp}P_B(u) \\
    & = \underbrace{(1-\lambda) u}_{\substack{\in \mathcal N(\lambda-P_AP_BP_A) \\ \neq 0}} - \quad \underbrace{P_{A^\perp}P_B(u)}_{\substack{\in \mathcal N(P_AP_BP_A) \\ \neq 0}}.
  \end{align}
  This, when combined with Lemma \ref{lem:uniqueness}, implies that $\Lambda(x, P_AP_BP_A) = \{0, \lambda\}$.

  We now show the last equality of \eqref{ex:Lambdax:eq1}. By \eqref{pr:proj3:SDv}, we have $u = w_- + w_+$, where $w_{\pm} := \frac 1 2 u \pm \frac {1}{2 \sqrt \lambda} P_B(u) \in \mathcal N(\mu_\pm - \frac{P_A+P_B}{2})$ are nonzero. Moreover,
  \begin{align}\label{} \nonumber
    P_B(u) & = \textstyle \left( \frac 1 2 P_B(u) + \frac{1}{2\sqrt \lambda} P_AP_B(u) \right)
           + \left( \frac 1 2 P_B(u) - \frac{1}{2\sqrt \lambda} P_AP_B(u) \right) \\
           &= \sqrt \lambda w_+ - \sqrt \lambda w_-,
  \end{align}
  so that $x = u - P_B(u) = (1 - \sqrt \lambda) w_+ + (1 + \sqrt \lambda) w_-$. This, when combined with Lemma \ref{lem:uniqueness}, implies that $\Lambda(x, \frac{P_A+P_B}2) = \{\mu_+,\mu_-\}$.
\end{proof}

\section{Rates of Convergence} \label{sec:rates}
 Let $A$ and $B$ be two nontrivial ($\neq \{0\}$, $\neq \mathbb R^d$) subspaces of $\mathbb R^d$.  We express the rates of $Q$-linear convergence for the MAP and the MSP using the eigenvalues of the operators $P_BP_AP_B$ and $\frac{P_A + P_B}{2}$.
\begin{theorem} \label{thm:main}
  Let the sequences $\{a_k\}_{k=0}^\infty$ and $\{x_k\}_{k=0}^\infty$ be defined by the MAP \eqref{int:ak} and the MSP \eqref{int:xk}, respectively, where $a_0 = x_0 \in \mathbb R^d$. Assume that $\|a_1 - P_{A \cap B}(a_0)\| > 0$ ( see Remarks \ref{rem:parallel} and \ref{rem:assumption}). Then, for each $k = 1,2, \ldots$, we have
  \begin{equation}\label{thm:main:xkak}
    \|a_k - P_{A \cap B}(a_0)\| > 0 \quad \text{and} \quad \|x_k - P_{A \cap B}(x_0)\| >0.
  \end{equation}
  Moreover,  for each $k = 0, 1, 2 , \ldots$, we have
  \begin{equation}\label{thm:main:ineq:ak}
    \|a_{k+1} - P_{A \cap B}(a_0)\| \leq
    \left\{
    \begin{array}{cc}
      \sqrt{\lambda}, & \text{if }\ k = 0 \\
      \lambda, & \text{if }\ k \geq 1
    \end{array}
    \right\} \cdot \|a_k - P_{A \cap B}(a_0)\|
  \end{equation}
  and
  \begin{equation}\label{thm:main:ineq:xk}
    \|x_{k+1} - P_{A \cap B}(x_0)\| \leq \mu \|x_k - P_{A \cap B}(x_0)\|,
  \end{equation}
  where  the rates of $Q$-linear convergence satisfy
  \begin{equation}\label{thm:main:lambda}
    \lambda := \lim_{k\to\infty} \frac{\|a_{k+1} - P_{A \cap B}(a_0)\|}{\|a_k - P_{A \cap B}(a_0)\|}
    = \max \Big\{\Lambda(a_0, P_BP_AP_B) \setminus \{0,1\} \Big\}
  \end{equation}
  and
  \begin{equation}\label{thm:main:mu}
    \mu := \lim_{k\to\infty} \frac{\|x_{k+1} - P_{A \cap B}(x_0)\|}{\|x_k - P_{A \cap B}(x_0)\|}
    = \max \Big\{\Lambda (\textstyle x_0, \frac{P_A+P_B}2) \setminus\{0,1\} \Big\}.
  \end{equation}
  In particular, $\lambda \in (0, \cos^2(\theta_F)]$ and $\mu \in [\frac 1 2 - \frac 1 2 \cos(\theta_F), \frac 1 2 + \frac 1 2 \cos(\theta_F)]$.
\end{theorem}

\begin{proof}
  First, we establish the above-mentioned properties for the sequence $\{a_k\}_{k=0}^\infty$. Using the spectral theorem (Theorem \ref{thm:SDT}), we have
  \begin{equation}\label{}
    a_0 = v_0 + v_1 + \sum_{\lambda \in \Lambda(P_BP_AP_B) \setminus \{0,1\}} v_{\lambda},
  \end{equation}
  where $ v_0 := P_{\mathcal N(P_BP_AP_B)}(a_0)$ and $v_{\lambda} := P_{\mathcal N(\lambda - P_BP_AP_B)}(a_0)$, $\lambda \in \Lambda(P_BP_AP_B) \setminus \{0\}$. Note that $v_1 = P_{A\cap B}(a_0)$. By combining the assumption $\|P_AP_B(a_0) - P_{A \cap B(a_0)}\| > 0$ with \eqref{cor:NullSets:S}, we see that $\|P_BP_AP_B(a_0) - P_{A \cap B}(a_0)\| > 0$. Using the orthogonality of the eigenvectors, we have
  \begin{equation}\label{}
    0 < \|P_BP_AP_B(a_0) - P_{A \cap B}(a_0)\|^2
    = \sum_{\lambda \in \Lambda(P_BP_AP_B) \setminus \{0,1\}} \lambda^2 \|v_{\lambda}\|^2.
  \end{equation}
  This shows that $\Lambda(a_0,P_BP_AP_B) \setminus \{0, 1\} \neq \emptyset$, say it consists of $\{\lambda_1 > \ldots > \lambda_m\}$ for some $m \geq 1$. On the other hand, using the identities:  $P_AP_B(v_0) = 0$ (see \eqref{lem:NullSet:S*S}), $P_B(v_{\lambda_i}) = v_{\lambda_i}$ and  $P_A(v_1) =  P_B(v_1)  = v_1$, for each $k =  1, 2, \ldots$, we have
  \begin{equation}\label{}
    a_k - P_{A\cap B}(a_0) = P_A(P_BP_AP_B)^{k-1} P_B  (a_0) - v_1
    = \sum_{i=1}^{m}\lambda_i^{k-1} P_A(v_{\lambda_i}).
  \end{equation}
  The vectors $P_A(v_{\lambda_i}) \in \mathcal N(\lambda_i - P_AP_BP_A)$ are nonzero, $i = 1,\ldots, m$, and pairwise orthogonal, as they belong to different eigenspaces; see Lemma \ref{lem:proj1}. Moreover, $\|P_A(v_{\lambda_i})\| = \sqrt{\lambda_i}\|v_{\lambda_i}\|$, $i=1,\ldots,m$; compare with \eqref{pr:proj1:norm}. Thus,  for $k = 1, 2, \ldots$,  we have
  \begin{equation}\label{pr:main:ak}
    \|a_k - P_{A\cap B}(a_0)\|^2 = \sum_{i=1}^{m} \lambda_i^{2(k-1)}\|P_A(v_{\lambda_i})\|^2
    = \sum_{i=1}^{m} \lambda_i^{2k-1}\|v_{\lambda_i}\|^2 > 0.
  \end{equation}
   Consequently, for $k = 0$, we obtain
  \begin{equation}\label{}
    \|a_1 - P_{A \cap B}(a_0)\|^2
    \leq \lambda_1 \sum_{i=1}^{m} \|v_{\lambda_i}\|^2
    \leq \lambda_1 \|a_0 - P_{A \cap B}(a_0)\|^2
  \end{equation}
  while for each $k = 1, 2, \ldots$, we arrive at
  \begin{equation}\label{}
    \|a_{k+1} - P_{A \cap B}(a_0)\|^2
    \leq \lambda_1^2 \sum_{i=1}^{m} \lambda_i^{2k-1}\|v_{\lambda_i}\|^2
    = \lambda_1^2 \|a_k - P_{A \cap B}(a_0)\|^2.
  \end{equation}
  This shows \eqref{thm:main:ineq:ak}. Furthermore, we see that
  \begin{equation}\label{}
    \frac{\|a_{k+1} - P_{A \cap B}(a_0)\|^2}{\|a_k - P_{A \cap B}(a_0)\|^2}
    = \frac{\sum_{i=1}^{m} \lambda_i^2 \left(\frac{\lambda_i}{\lambda_1}\right)^{2k-1}\|v_{\lambda_i}\|^2}{\sum_{i=1}^{m} \left(\frac{\lambda_i}{\lambda_1}\right)^{2k-1}\|v_{\lambda_i}\|^2} \to \lambda_1^2
  \end{equation}
  as $k \to \infty$, since $\frac{\lambda_i}{\lambda_1} < 1$ except for $i = 1$. This proves \eqref{thm:main:lambda}.

  \bigskip
  We now turn our attention to the properties of the sequence $\{x_k\}_{k=0}^\infty$. By \eqref{cor:NullSets:T}, we have $\|x_1 - P_{A \cap B}(x_0)\| > 0$. By using an analogous argument to the one presented above, we deduce that $\Lambda(x, \frac{P_A+P_B}{2}) \setminus \{0,1\} \neq \emptyset$, say it consists of $\{\mu_1 > \ldots > \mu_n\}$ for some $n \geq 1$. Using Theorem \ref{thm:SDT}, we can write
  \begin{equation}\label{}
    x_0 = w_0 + w_1 + \sum_{i=1}^{n} w_{\mu_i},
  \end{equation}
  where $w_0 := P_{\mathcal N(\frac{P_A+P_B}2)}(x_0)$, $w_{\mu_i} := P_{\mathcal N(\mu_i - \frac{P_A+P_B}2)}(x_0) \neq 0$, $i = 1, \ldots, n$, where again $w_1 = P_{A\cap B}(x_0)$. Knowing that $w_{\mu_i}$ are pairwise orthogonal, we have
  \begin{equation}\label{}
    \|x_k - P_{A \cap B}(x_0)\|^2 = \sum_{i=1}^{n} \mu_i^{2k} \|w_{\mu_i}\|^2 > 0.
  \end{equation}
   Consequently, for each $k = 0, 1, 2, \ldots$, we obtain
  \begin{equation}\label{}
    \|x_{k+1} - P_{A \cap B}(x_0)\|^2
    \leq \mu_1^2 \sum_{i=1}^{n} \mu_i^{2k}\|w_{\mu_i}\|^2
    = \mu_1^2 \|x_k - P_{A \cap B}(x_0)\|^2,
  \end{equation}
  which shows \eqref{thm:main:ineq:xk}. Moreover,
  \begin{equation}\label{}
    \frac{\|x_{k+1} - P_{A \cap B}(a_0)\|^2}{\|x_k - P_{A \cap B}(a_0)\|^2}
    = \frac{\sum_{i=1}^{m} \mu_i^2 \left(\frac{\mu_i}{\mu_1}\right)^{2k}\|w_{\mu_i}\|^2}{\sum_{i=1}^{m} \left(\frac{\mu_i}{\mu_1}\right)^{2k}\|w_{\mu_i}\|^2} \to \mu_1^2
  \end{equation}
  as $k \to \infty$. This proves \eqref{thm:main:mu} and completes the proof of the theorem.
\end{proof}

\begin{remark} \label{rem:cos12}

\begin{enumerate}[(i)]
  \item  Assume that $\theta_F > \frac \pi 3$ (equivalently, $\cos(\theta_F) < 1/2$).  Then the rates $\lambda$ and $\mu$ defined in \eqref{thm:main:lambda} and \eqref{thm:main:mu} satisfy $\mu \geq \frac 1 2 - \frac 1 2 \cos(\theta_F) > \cos^2(\theta_F) \geq \lambda$ and we conclude that the MAP is significantly faster than the MSP. In particular, starting from some point $k$ onward, we will have $\|x_k - P_{A \cap B}(x_0)\| > \|a_k - P_{A \cap B}(a_0)\|$.

  \item When  $\theta_F < \frac \pi 3$ (equivalently, $\cos(\theta_F) > 1/2$),  we can always find starting points for which the MSP outperforms the MAP. Indeed, is suffices to take any nonzero $a_0 = x_0 = w \in \mathcal N((\frac 1 2 - \frac 1 2 \cos(\theta_F)) - \frac{P_A+P_B}{2})$. Clearly, $\Lambda(x_0, \frac{P_A+P_B}2) = \{\frac 1 2 - \frac 1 2 \cos(\theta_F)\}$. Moreover, $\Lambda(a_0, P_BP_AP_B) = \{0, \cos^2(\theta_F)\}$ (write $a_0 = P_{B^\perp}(w) + P_B(w)$ and use Lemma \ref{lem:proj2}). Therefore, in this case, the rates $\lambda$ and $\mu$ defined in \eqref{thm:main:lambda} and \eqref{thm:main:mu} satisfy $\mu = \frac 1 2 - \frac 1 2 \cos(\theta_F) < \cos^2(\theta_F) = \lambda$. In fact, by \eqref{pr:proj2:PAPBw} and since $\|P_A(w)\| = \mu \|w\|$ (see \eqref{lem:orthogonality:2}), for each $k = 1,2, \ldots,$ we obtain
      \begin{align}\label{rem:cos12:ineq} \nonumber
        \|a_k - P_{A \cap B}(a_0)\|^2 & = \|(P_AP_B)^{k} (w)\|^2 = \lambda^{2k-1} \|P_A(w)\|^2\\
        &  = \lambda^{2k-1} \mu \|w\|^2 > \mu^{2k}\|w\|^2 = \|x_k - P_{A \cap B}(x_0)\|^2.
      \end{align}
      Note that for $\cos(\theta_F) = \frac 1 2$, we have $\lambda = \mu$ and \eqref{rem:cos12:ineq} holds as an equality.
\end{enumerate}
\end{remark}

\begin{remark}[Alternative for MAP]
It follows from Theorem \ref{thm:main} that the sequence $\{a_k\}_{k=0}^\infty$ defined by the MAP will either reach the solution after one iteration or that it will not reach the solution after any finite number of steps. Surprisingly, the same alternative holds true when $A$ and $B$ are both closed and convex subsets of $\mathbb R^d$ having nonempty intersection; see \cite[Theorem 7]{LukeTeboulleThao2020}.
\end{remark}

\begin{remark}[Comparing with \cite{KayalarWeinert1988} and \cite{ReichZalas2017}] \label{rem:KW88andRZ17} Let the sequences $\{a_k\}_{k=0}^\infty$ and $\{x_k\}_{k=0}^\infty$ be defined as in Theorem \ref{thm:main}. It follows from \eqref{int:KW88:eq} and \eqref{int:RZ17:eq}, that correspond to \cite{KayalarWeinert1988} and \cite{ReichZalas2017}, respectively, that
\begin{equation}\label{}
  \|a_k - P_{A \cap B}(a_0)\| \leq \cos^{2k-1}(\theta_F) \|a_0\|
\end{equation}
and
\begin{equation}\textstyle
  \|x_k - P_{A \cap B}(x_0)\| \leq \left(\frac 1 2 + \frac 1 2 \cos(\theta_F) \right)^k \|x_0\|.
\end{equation}
In particular, we obtain $R$-linear convergence. On the other hand, using Theorem \ref{thm:main}, we arrive at $Q$-linear convergence. Moreover,
\begin{equation}\label{rem:KW88andRZ17:inek:ak}
    \|a_{k+1} - P_{A \cap B}(a_0)\| \leq
    \left\{
    \begin{array}{cc}
      \cos(\theta_F), & \text{if }\ k = 0 \\
      \cos^2(\theta_F), & \text{if }\ k \geq 1
    \end{array}
    \right\} \cdot \|a_k - P_{A \cap B}(a_0)\|
  \end{equation}
and
\begin{equation}\label{rem:KW88andRZ17:inek:xk} \textstyle
  \|x_{k+1} - P_{A \cap B}(x_0)\| \leq \left(\frac 1 2 + \frac 1 2 \cos(\theta_F) \right) \|x_k - P_{A \cap B}(x_0)\|,
\end{equation}
$k = 0, 1, 2, \ldots$. Inequalities \eqref{rem:KW88andRZ17:inek:ak} and \eqref{rem:KW88andRZ17:inek:xk} appear to be less known, if not new.
\end{remark}

\begin{remark}[Directional Asymptotics]
  By slightly adjusting the argument in the proof of Theorem \ref{thm:main}, we see that
  \begin{equation}\label{rem:DirAsy:eq1}
    \lim_{k \to \infty} \frac{a_k - P_{A \cap B}(a_0)}{\|a_k - P_{A \cap B}(a_0)\|}
    = \lim_{k \to \infty} \frac{a_k - a_{k+1}}{\|a_k - a_{k+1}\|} = \frac{P_A(  v_\lambda )}{\|P_A( v_\lambda )\|},
  \end{equation}
  where  $\lambda$ is given by \eqref{thm:main:lambda} and $v_\lambda := P_{\mathcal N(\lambda - P_BP_AP_B)}(a_0)$.  Similarly, we obtain
  \begin{equation}\label{rem:DirAsy:eq2}
    \lim_{k \to \infty} \frac{x_k - P_{A \cap B}(x_0)}{\|x_k - P_{A \cap B}(x_0)\|}
    = \lim_{k \to \infty} \frac{x_k - x_{k+1}}{\|x_k - x_{k+1}\|} = \frac{ w_\mu }{\| w_\mu \|},
  \end{equation}
  where  $\mu$ is given by \eqref{thm:main:mu} and where $w_\mu := P_{\mathcal N(\mu - \frac{P_A+P_B}{2})}(x_0)$.  The above-mentioned directional asymptotics were of particular interest in \cite[Theorem 3.1]{BauschkeLalXianfu2023}, in which the authors considered general Fej\'{e}r monotone sequences.  The limits in \eqref{rem:DirAsy:eq1}--\eqref{rem:DirAsy:eq2} are also related to the so-called \emph{power iteration} which is used to find eigenpairs of diagonalizable matrices; see, for instance, \cite[Example 7.3.7]{Meyer2000}.
\end{remark}

We are now ready to present our main result.

\begin{theorem} \label{thm:main2}
  Let the sequences $\{a_k\}_{k=0}^\infty$ and $\{x_k\}_{k=0}^\infty$ be defined by the MAP \eqref{int:ak} and the MSP \eqref{int:xk}, respectively, where $a_0 = x_0 \in A \cup B$. Assume that $\|a_1 - P_{A \cap B}(a_0)\| > 0$  ( see Remarks \ref{rem:parallel} and \ref{rem:assumption}).  Then, for each $k = 1,2, \ldots$, we have
  \begin{equation}\label{thm:main2:xkak}
    \|x_k - P_{A \cap B}(x_0)\| > \|a_k - P_{A \cap B}(a_0)\| > 0
  \end{equation}
  and the rates of $Q$-linear convergence $\lambda$ and $\mu$ defined in \eqref{thm:main:lambda} and \eqref{thm:main:mu}, respectively, satisfy
  \begin{equation}\label{thm:main2:lambdaANDmu} \textstyle
    \mu = \frac 1 2 + \frac 1 2 \sqrt{\lambda} > \lambda.
  \end{equation}
  In particular, for the above-described set of starting points, the MAP outperforms the MSP.
\end{theorem}

\begin{proof}
  The second part of the theorem, that is, equality \eqref{thm:main2:lambdaANDmu}, follows directly from Theorem \ref{thm:Lambdax} and Theorem \ref{thm:main}. We now turn our attention to inequality \eqref{thm:main2:xkak}. We borrow the notation from the proof of Theorem \ref{thm:Lambdax} with $x := a_0 = x_0$.

  We assume that $x \in A$. Then, by \eqref{pr:Lambdax:xA}, we have
  \begin{equation} \label{pr:main2:xA}
    x = w_{\frac 1 2} + \sum_{i=1}^m (\underbrace{w_i^- + w_i^+}_{u_{\lambda_i}}) + w_1,
  \end{equation}
  with $w_i^+$ and $w_i^-$ defined as in \eqref{pr:Lmabdax:ui},  and where $w_{\frac 1 2} \in A \cap B^\perp$.  Noticing that $a_k = (P_AP_BP_A)^k(x)$, $k = 1,2,\ldots,$ and $P_{A \cap B}(a_0) = w_1$, and using the orthogonality of the eigenvectors $u_{\lambda_i}$,  for each $k = 1,2, \ldots$,  we have
  \begin{equation}\label{pr:main2:ak}
    \|a_k - P_{A \cap B}(a_0)\|^2 = \sum_{i=1}^{m} \lambda_i^{2k} \|u_{\lambda_i}\|^2.
  \end{equation}
  On the other hand, by \eqref{lem:orthogonality:3}, we get $\|w_i^\pm\|^2 = \mu_i^\pm \|u_{\lambda_i}\|^2$, $i = 1, \ldots, m$. Consequently, thanks to the orthogonality of the eigenvectors $w_i^\pm$,  for each $k = 1,2, \ldots$,  we get
  \begin{align}\label{pr:main2:xk} \nonumber
    \|x_k - P_{A \cap B}(x_0)\|^2 & = {\textstyle(\frac 1 2)^{2k}}\|w_{\frac 1 2}\|^2 + \sum_{i=1}^{m} \left((\mu_i^-)^{2k} \|w_i^-\|^2 + (\mu_i^+)^{2k} \|w_i^+\|^2 \right)\\
    & = {\textstyle(\frac 1 2)^{2k}}\|w_{\frac 1 2}\|^2 + \sum_{i=1}^{m} \left((\mu_i^-)^{2k+1} + (\mu_i^+)^{2k+1} \right) \|u_{\lambda_i}\|^2.
  \end{align}
  Observe that for $t := \mu_i^+ \in (\frac 1 2, 1)$, we have $\mu_i^- = 1-t$ and $ \sqrt{\lambda_i} = \mu_i^+ - \mu_i^-  = 2t-1$. Thus, using the  strict convexity  of $s \mapsto s^{2k}$, we get
  \begin{align}\label{pr:main2:ineqMu} \nonumber
    (\mu_i^+)^{2k+1} + (\mu_i^-)^{2k+1}
    & = t(\mu_i^+)^{2k} + (1-t)(-\mu_i^-)^{2k}
     > (t \mu_i^+ + (1-t) (-\mu_i^-))^{2k}  \\
    & = (t^2 - (1-t)^2)^{2k}
    = (2t - 1)^{2k}
    = \lambda_i^k > \lambda_i^{2k}.
  \end{align}
  By combining \eqref{pr:main2:ak} and \eqref{pr:main2:xk} with \eqref{pr:main2:ineqMu}, we arrive at \eqref{thm:main2:xkak}.

  \bigskip
  We now assume that $x \in B$. By \eqref{pr:Lambdax:xB}, we can again write
  \begin{equation} \label{pr:main2:xB}
    x = w_{\frac 1 2} + \sum_{i=1}^m (\underbrace{w_i^- + w_i^+}_{v_{\lambda_i}}) + w_1,
  \end{equation}
  where this time $w_i^+$ and $w_i^-$ are defined as in \eqref{pr:Lmabdax:vi},  and where $w_{\frac 1 2} \in A^\perp \cap B$. The main difference  from the previous case is that now  $a_k =  P_A(P_BP_AP_B)^{k-1}P_B(x) $, $k = 1,2,\ldots$,  as in the proof of Theorem \ref{thm:main}.  By repeating the argument used for deriving \eqref{pr:main:ak}  (the orthogonality of the eigenvectors $P_A(v_{\lambda_i})$ of $P_AP_BP_A$  and $\|P_A(v_{\lambda_i})\| = \sqrt{\lambda_i} \|v_{\lambda_i}\|$),  we get
  \begin{equation}\label{pr:main2:akB}
    \|a_k - P_{A \cap B}(a_0)\|^2 = \sum_{i=1}^{m} \lambda_i^{2(k-1)} \|P_A(v_{\lambda_i})\|^2
    = \sum_{i=1}^{m} \lambda_i^{2k-1} \|v_{\lambda_i}\|^2;
  \end{equation}
  Moreover, by adjusting \eqref{lem:orthogonality:3} to the eigenvectors $v_{\lambda_i}$ of $P_BP_AP_B$, we obtain $\|w_i^\pm\|^2 = \mu_i^\pm \|v_{\lambda_i}\|^2$ and
  \begin{equation} \label{pr:main2:xkB}
    \|x_k - P_{A \cap B}(x_0)\|^2
     = {\textstyle(\frac 1 2)^{2k}}\|w_{\frac 1 2}\|^2 + \sum_{i=1}^{m} \left((\mu_i^-)^{2k+1} + (\mu_i^+)^{2k+1} \right) \|v_{\lambda_i}\|^2;
  \end{equation}
  compare with \eqref{pr:main2:xk}. By \eqref{pr:main2:ineqMu}, we have $(\mu_i^+)^{2k+1} + (\mu_i^-)^{2k+1}  > \lambda_i^k \geq  \lambda_i^{2k-1}$, which, when combined with \eqref{pr:main2:akB} and \eqref{pr:main2:xkB}, yields \eqref{thm:main2:xkak}.
\end{proof}

\subsection{A Word on Compactness}
It might be a natural approach to extend the results of  Theorems \ref{thm:main} and \ref{thm:main2}  assuming that  $A$ and $B$ are both closed linear subspaces of a Hilbert space $\mathcal H$ and that both operators $P_BP_AP_B$ and $\frac{P_A+P_B}2$ are compact.  However, if $P_BP_AP_B$ were compact with an infinite number of eigenvalues, say, $\Lambda(P_BP_AP_B)\setminus \{0\} = \{\lambda_1 > \lambda_2 > \ldots \}$, then, by repeating the argument in the proof of Lemma \ref{lem:proj3}, we would arrive at $\mu_n := \frac 1 2 + \frac 1 2 \sqrt{\lambda _n} \in \Lambda(\frac{P_A+P_B}2) \setminus\{0\}$. Noting that $\lambda_n \to 0$, we would also have $\mu_n \to \frac 1 2$. This would imply that $\frac{P_A+P_B}2$ cannot be compact as zero is the only possible accumulation point for the spectrum. In fact, we have the following lemma.

\begin{lemma}[Lack of Compactness] \label{ex:TisNotCompact}
  Let $A$ and $B$ be closed linear subspaces of $\mathcal H$. Assume that one of the subspaces among $A$ and $B$ is of finite dimension while the second one has infinite dimension. Then the symmetric alternating projection $S := P_BP_AP_B$ is compact (since $\dim \mathcal R(S) < \infty$) while the simultaneous projection $T := \frac{P_A+P_B}{2}$ is not compact.
\end{lemma}

\begin{proof}
   Assume that $\dim A = \infty$ and $\dim B < \infty$. Observe that $P_A$ is not compact. Indeed, recall that the range of a compact operator contains no closed infinite-dimensional subspace; see, for example, \cite[Theorem 2.5]{FillmoreWilliams1971}. Moreover, in our case $\mathcal R(P_A) = A$. Thus $P_A$ is not compact, as claimed. Consequently, if $T$ were compact, so would $2T - P_B = P_A$ be - a contradiction.
\end{proof}

This implies that in the general infinite-dimensional case,  for the proofs of Theorems \ref{thm:main} and \ref{thm:main2},  we have to follow a different path, which will be presented in a sequel to the present paper.

\section*{Acknowledgement}
\addcontentsline{toc}{section}{Acknowledgement}
We are grateful to two anonymous referees for their close reading of our manuscript, and for their detailed comments and constructive suggestions.

\section*{Funding}
\addcontentsline{toc}{section}{Funding}
This work was partially supported by the Israel Science Foundation (Grants 389/12 and 820/17), the Fund for the Promotion of Research at the Technion (Grant 2001893) and by the Technion General Research Fund (Grant 2016723).

\section*{Appendix A} \label{sec:AppendixA}
\addcontentsline{toc}{section}{Appendix A}

\begin{proof}[Proof of Lemma \ref{lem:principal0}]
  Note that if $A \cap B \neq \{0\}$, then $\cos(\theta_D) = 1 = \|P_AP_B\|$. Moreover, if $A \cap B = \{0\}$, then, thanks to \eqref{int:KW88:eq}, we again have $\cos(\theta_D) = \cos(\theta_F) = \|P_AP_B\| $. Clearly, $\|P_AP_B\|^2 = \|(P_AP_B)^*(P_AP_B)\| = \lambda^*$. This shows the equality $\cos(\theta_D) = \sqrt{\lambda^*}$.

  Suppose that the unit vectors $u^* \in A$ and $v^* \in B$ satisfy $\langle u^*, v^*\rangle = \sqrt{\lambda^*}$. If $\lambda^* = 0$, then, for all $u \in A$, $v \in B$, we get $|\langle u^*, v\rangle| = |\langle u, v^*\rangle| \leq \sqrt{\lambda^*} = 0$. Thus, $u^* \in B^\perp$ and $v^* \in A^\perp$. In particular, equality \eqref{lem:principal0:reciprocal} holds. We may thus assume that $\lambda^* > 0$. Then,
  \begin{equation}\label{}
    \sqrt{\lambda^*} = \langle u^*, v^*\rangle
    \leq \max_{u \in A,\ \|u\| = 1} \langle u, v^*\rangle
    \leq \cos(\theta_D) = \sqrt{\lambda^*}.
  \end{equation}
  Moreover, using the Riesz representation theorem, we see that
  \begin{equation}\label{}
    \|P_A(v^*)\| = \max_{\|u\| = 1} \langle u, P_A(v^*)\rangle
    = \max_{\|u\| = 1} \langle P_A(u), v^*\rangle
    = \max_{u \in A,\ \|u\| = 1} \langle u, v^*\rangle.
  \end{equation}
  Consequently, $\|P_A(v^*)\| = \sqrt{\lambda^*}$. Using the equality $\langle u^*, P_A(v^*)\rangle = \langle u^*, v^*\rangle = \sqrt{\lambda^*}$, we obtain
  \begin{equation}\label{}
    \|P_A(v^*) - \sqrt{\lambda^*} u^*\|^2 = \|P_A(v^*)\|^2 - 2 \sqrt{\lambda^*} \langle u^*, P_A(v^*)\rangle + \lambda^*\|u^*\|^2 = 0,
  \end{equation}
  which shows $P_A(v^*) = \sqrt{\lambda^*} u^*$. Similarly we can show that $P_B(u^*) = \sqrt{\lambda^*} v^*$.
\end{proof}

\begin{proof}[Proof of Lemma \ref{lem:principal}.]
  Let $U_n := \spn\{u_0, \ldots, u_n\}$ and $V_n := \spn\{v_0, \ldots, v_n\}$, $n = 0, 1 ,\ldots,p$. Also, put $A_n := A \cap U_{n-1}^\perp$ and $B_n := B \cap V_{n-1}^\perp$, $n = 1, \ldots, p$. Then \eqref{def:principal:cos} can be written as $\langle u_n, v_n \rangle = \cos(\theta_D(A_n, B_n))$. Moreover, thanks to Lemma \ref{lem:principal0}, we know that $\cos(\theta_D(A_n, B_n) = \sqrt{\lambda_n}$, where $\lambda_n := \max \Lambda(P_{B_n}P_{A_n}P_{B_n})$, $n = 1, \ldots, p$.

  \bigskip
  \emph{\underline{Step 1.}}
  Assume that (i) holds. Thanks to Lemma \ref{lem:principal0}, we obtain $P_{A_n}(v_n) = \sqrt{\lambda_n} u_n$ and $P_{B_n}(u_n) = \sqrt{\lambda_n} v_n$. Observe that $A_1 = A$ and $B_1 = B$ which shows \eqref{lem:principal:reciprocal} for $n = 1$. Suppose now that \eqref{lem:principal:reciprocal} holds for $u_i$ and $v_i$, $i = 1, \ldots, n-1$, $n \leq p$. Then, using mathematical induction, we have
  \begin{equation}\label{pr:principal:viun}
    \langle v_i, u_n \rangle = \langle v_i, P_A(u_n) \rangle = \langle P_A(v_i), u_n \rangle
    = \sqrt{\lambda_i} \langle u_i, u_n \rangle = 0,
  \end{equation}
  which shows that $u_n \in V_{n-1}^\perp$. Noting that $V_{n-1} \subset B$ and using Fact \ref{fact:commutingProj}, we get $P_{B_n} = P_B P_{V_{n-1}^\perp}$. Consequently, $P_B(u_n) = P_B P_{V_{n-1}^\perp}(u_n) = P_{B_n}(u_n)$. Analogously, we can show that $P_A(v_n) = P_AP_{V_{n-1}^\perp}(v_n) = P_{A_n}(v_n)$. This yields \eqref{lem:principal:reciprocal} for $n = 1, \ldots, p$.

  We now show that $\lambda_n > 0$ for $n = 1, \ldots, r$. Indeed, since $n \leq r$, we can always choose an eigenvalue $\lambda >0$ and a corresponding unit vector $v \in \mathcal N(\lambda - P_BP_AP_B)$ such that $v \in V_{n-1}^\perp$. Then
  \begin{equation}\label{}
    \|P_A(v)\|^2 = \|P_AP_B(v)\|^2
    = \langle P_BP_AP_B(v), v\rangle
    = \lambda > 0
  \end{equation}
  and we can define $u := \frac{P_A(v)}{\|P_A(v)\|}$. Moreover, by \eqref{lem:principal:reciprocal}, for each $i = 1, \ldots, n-1$, we get
  \begin{equation}\label{}
    \langle u_i, u \rangle = \frac{\langle u_i, P_A(v)\rangle}{\|P_A(v)\|}
    = \frac{\langle u_i, u\rangle}{\sqrt{\lambda}}
    = \frac{\langle u_i, P_B(v)\rangle}{\sqrt{\lambda}}
    = \frac{\langle P_B(u_i), v\rangle}{\sqrt{\lambda}}
    = \frac{\sqrt{\lambda_i}}{\sqrt{\lambda}}\langle v_i, v\rangle = 0,
  \end{equation}
  which shows that $u \in U_{n-1}^\perp$. Consequently, $u \in A_n$, $v \in B_n$ and $\sqrt{\lambda_n} \geq \langle u, v\rangle = \|P_A(v)\| > 0$.

  Observe that we also have $\lambda_n = 0$ for $n = r+1, \ldots p$. Otherwise, because of \eqref{lem:principal:reciprocal}, the vector $v_n$ would become an eigenvector of $P_BP_AP_B$ that is orthogonal to $V_r$ and so $\rank (P_BP_AP_B) \geq r+1$ -- a contradiction.

  \bigskip
  \emph{\underline{Step 2.}}
  Assume that (ii) holds. In view of Lemma \ref{lem:principal0}, it suffices to show that $\Lambda_\#(P_{B_n}P_{A_n}P_{B_n}) \setminus \{0\} = \{\lambda_n, \ldots,\lambda_r\}$.
  Indeed, similarly to \eqref{pr:principal:viun}, for each $i \neq j$, we get $\langle v_i, u_j\rangle = \sqrt{\lambda_i} \langle u_i, u_j \rangle = 0$. Thus for $i \geq n$, we have $u_i \in V_{n-1}^\perp$ and $v_i \in U_{n-1}^\perp$. Consequently, $P_{A_n}(v_i) = \sqrt{\lambda_i}u_i$ and $P_{B_n}(u_i) = \sqrt{\lambda_i} v_i$. In particular, $v_i \in \mathcal N(\lambda_i - P_{B_n}P_{A_n}P_{B_n})$ for $i = n, \ldots, r$ and $\{\lambda_n, \ldots,\lambda_r\} \subset \Lambda_\#(P_{B_n}P_{A_n}P_{B_n}) \setminus \{0\}$.

  Suppose there is a unit vector $v \in \mathcal N(\lambda - P_{B_n}P_{A_n}P_{B_n})$ for some $\lambda > 0$ that is orthogonal to $v_n, \ldots, v_r$. Note that $v$ is also orthogonal to $v_1, \ldots, v_{n-1}$ because $v \in B_n$. By \eqref{lem:principal:reciprocal}, for each $i = 1,\ldots,r$, we have
  \begin{equation}\label{}
    \langle v, u_i \rangle = \langle P_B(v), u_i \rangle = \langle v, P_B(u_i) \rangle
    = \sqrt{\lambda_i} \langle v, v_i \rangle = 0.
  \end{equation}
  Consequently, $v \in U_{n-1}^\perp$ and $P_A(v) = P_{A_n}(v)$. Similarly, for $u := \frac{P_A(v)}{\|P_A(v)\|}$, we have $\|P_A(v)\| = \sqrt{\lambda}$ and
  \begin{equation}\label{}
    \langle u, v_i \rangle = \frac{\langle P_A(v), u_i \rangle}{\sqrt{\lambda}} = \frac{\langle v, u_i \rangle}{\sqrt{\lambda}} = 0,
  \end{equation}
  $i = 1,\ldots,r$. Thus $u \in V_{n-1}^\perp$ and $P_B(u) = P_{B_n}(u)$. Consequently, we obtain
  \begin{align}\label{} \nonumber
    P_BP_AP_B(v) & = P_BP_A(v) = \sqrt{\lambda} P_B(u) = \sqrt{\lambda} P_{B_n}(u) \\
    & = P_{B_n}P_A(v) = P_{B_n}P_{A_n}(v) = P_{B_n}P_{A_n}P_{B_n}(v) = \lambda v,
  \end{align}
  that is, $v$ is an eigenvector of $P_BP_AP_B$ which is orthogonal to $v_1,\ldots,v_r$. This implies that $\rank(P_BP_AP_B) \geq r+1$ -- a contradiction with our assumption. Consequently we must have $\Lambda_\#(P_{B_n}P_{A_n}P_{B_n}) \setminus \{0\} = \{\lambda_n, \ldots,\lambda_r\}$, as asserted.
\end{proof}

\small
%\footnotesize
%\bibliographystyle{siamplain}
%\bibliography{references}

\end{document}